\documentclass[12]{amsart}
\usepackage{mathrsfs}
\usepackage{stmaryrd}
\usepackage[latin1]{inputenc}
\usepackage{amsmath}
\usepackage{amscd}
\usepackage{amsfonts}
\usepackage{amssymb}
\usepackage[all,cmtip]{xy}
\usepackage{verbatim}   

\usepackage{amsthm}

\theoremstyle{plain}
\newtheorem{thm}{Theorem}[subsection] 
\newtheorem{defn}[thm]{Definition}
\newtheorem{prop}[thm]{Proposition}
\newtheorem{conj}[thm]{Conjecture}
\newtheorem{lemma}[thm]{Lemma}
\newtheorem{corollary}[thm]{Corollary}

\theoremstyle{definition}
\newtheorem{rem}[thm]{Remark}

\title{GALOIS LATTICES AND STRONGLY DIVISIBLE LATTICES IN THE UNIPOTENT CASE}
\author{HUI GAO}
 
\address{Beijing International Center for Mathematical Research, Peking University, No. 5 Yiheyuan Road, Haidian District, Beijing 100871, China}
\email{gaohui@math.pku.edu.cn}


\begin{document}
\maketitle

\pagestyle{myheadings}
\markright{GALOIS LATTICES AND STRONGLY DIVISIBLE LATTICES}


\newcommand{\Zp}{\mathbb{Z}_p}
\newcommand{\Qp}{\mathbb{Q}_p}
\newcommand{\Z}{\mathbb{Z}}
\newcommand{\Q}{\mathbb{Q}}
\newcommand{\R}{\mathbb{R}}
\newcommand{\C}{\mathbb{C}}
\newcommand{\N}{\mathbb{N}}

\newcommand{\Spec}{\textnormal{Spec}}
\newcommand{\Hom}{\textnormal{Hom}}
\newcommand{\Frac}{\textnormal{Frac}}
\newcommand{\Fil}{\textnormal{Fil}}
\newcommand{\Mod}{\textnormal{Mod}}
\newcommand{\ModFI}{\textnormal{ModFI}}
\newcommand{\Rep}{\textnormal{Rep}}
\newcommand{\Gal}{\textnormal{Gal}}
\newcommand{\Ker}{\textnormal{Ker}}
\newcommand{\HT}{\textnormal{HT}}
\newcommand{\Mat}{\textnormal{Mat}}
\newcommand{\GL}{\textnormal{GL}}

\newcommand{\Zl}{\mathbb{Z}_{\ell}}
\newcommand{\Ql}{\mathbb{Q}_{\ell}}
\newcommand{\A}{\mathbb{A}}

\newcommand{\bigM}{\mathcal{M}}
\newcommand{\bigD}{\mathcal{D}}
\newcommand{\huaS}{\mathfrak{S}}
\newcommand{\huaM}{\mathfrak{M}}
\newcommand{\Acris}{A_{\textnormal{cris}}}
\newcommand{\bigO}{\mathcal{O}}

\newcommand{\bigN}{\mathcal{N}}
\newcommand{\huaN}{\mathfrak{N}}

\newcommand{\M}{\mathcal{M}}
\newcommand{\BdR}{B_{\textnormal{dR}}}
\newcommand{\bigOE}{\mathcal{O}_{\mathcal{E}}}

\newcommand{\bolda}{\boldsymbol{\alpha}}
\newcommand{\boldb}{\boldsymbol{\beta}}
\newcommand{\bolde}{\boldsymbol{e}}
\newcommand{\boldf}{\boldsymbol{f}}
\newcommand{\boldd}{\boldsymbol{d}}

\newcommand{\tildeS}{\tilde{S}}
\newcommand{\tS}{\tilde{S}}
\newcommand{\tildeM}{\tilde{M}}
\newcommand{\tM}{\tilde{M}}

\newcommand{\Asthat}{\widehat{A_{\textnormal{st}}}}
\newcommand{\D}{\mathcal{D}}

\newcommand{\MF}{MF^{(\varphi, N)}}
\newcommand{\MFwa}{MF^{(\varphi, N)-\textnormal{w.a.}}}
\newcommand{\bigMF}{\mathcal{MF}^{(\varphi, N)}}
\newcommand{\bigMFwa}{\mathcal{MF}^{(\varphi, N)-\textnormal{w.a.}}}

\newcommand{\Mdual}{\huaM^{\vee}}
\newcommand{\StarhuaM}{\huaS\otimes_{\varphi,\huaS} \huaM}
\newcommand{\ModhuaS}{\Mod_{\huaS}^{\varphi}}

\newcommand{\Ghat}{\hat{G}}
\newcommand{\Rhat}{\hat{\mathcal{R}}}
\newcommand{\mhat}{\hat{\huaM}}
\newcommand{\That}{\hat{T}}
\newcommand{\ihat}{\hat{\iota}}
\newcommand{\tn}{t^{\{n\}}}
\newcommand{\huat}{\mathfrak{t}}

\begin{abstract}
Let $p$ be a prime. We prove that there is an anti-equivalence between the category of unipotent strongly divisible lattices of weight $p-1$ and the category of Galois stable $\Zp$-lattices in unipotent semi-stable representations with Hodge-Tate weights $\subseteq \{0, \ldots, p-1\}$. This completes the last remaining piece of Breuil's conjecture(Conjecture 2.2.6 in \cite{Bre02}).
\end{abstract}

\tableofcontents 

\section*{Introduction}

Let $p$ be a prime, $k$ a perfect field of characteristic $p$, $W(k)$ the ring of Witt vectors, $K_0 = W(k)[\frac{1}{p}]$ the fraction field, $K$ a finite totally ramified extension of $K_0$, $e =e(K/K_0)$ the ramification index and $G=G_K =\Gal(\overline{K}/K)$ the absolute Galois group. The idea of $p$-adic Hodge theory is to classify semi-stable $p$-adic Galois representations by some ``linear algebra data", in particular, weakly admissible filtered $(\varphi, N)$-modules (see Section 1). Since $G_K$ is compact, thus for any continuous $p$-adic representation $\rho: G_K \to \GL_d(\Qp)$, there always exist $G_K$-stable $\Zp$-lattices, and integral $p$-adic Hodge theory studies these lattices in semi-stable $p$-adic Galois representations. A natural question is to ask whether there exist corresponding integral structures in the ``linear algebra data" side. In \cite{FL82}, Fontaine and Laffaille defined some $W(k)$-lattices in the filtered $(\varphi, N)$-modules, but unfortunately, it only works for $e=1$ and when $\rho$ is a crystalline representation with Hodge-Tate weights in $\{0, \ldots, p-1\}$.
Later Breuil introduced the theory of filtered $(\varphi, N)$-modules over  $S$ (\cite{Bre97}, \cite{Bre98}, \cite{Bre99a}) to study semi-stable Galois representations (e.g., see Section 3), and showed that the category of such modules is equivalent to the category of filtered $(\varphi, N)$-modules. Furthermore, Breuil defined a natural integral structure in these modules, which are called strongly divisible lattices. There is a functor $T_{\text{st}}$ which will send a strongly divisible lattice to a $G_K$-stable $\Zp$-lattice in the semi-stable Galois representation. In \cite{Bre02}, Breuil proposed a classification of $G_K$-stable $\Zp$-lattices in semi-stable representations via strongly divisible lattices.

\begin{conj}[\cite{Bre02}]
\begin{enumerate}
\item For $0 \leq r < p-1$, $T_{\textnormal{st}}$ induces an anti-equivalence between the category of strongly divisible lattices of weight $r$ and the category of $G_K$-stable $\Zp$-lattices in semi-stable representations of $G_K$ with Hodge-Tate weights $\subseteq \{0, \ldots, r\}$.

\item For $r =p-1$, $T_{\textnormal{st}}$ induces an anti-equivalence between the category of unipotent strongly divisible lattices of weight $r$ and the category of $G_K$-stable $\Zp$-lattices in unipotent semi-stable representations of $G_K$ with Hodge-Tate weights $\subseteq \{0, \ldots, r\}$.
\end{enumerate}
\end{conj}

Here, \emph{``unipotency"} is a technical condition (see Definition 1.0.7), and the terminology comes from the link with $p$-divisible groups, see the remark following Definition 2.1.1 of \cite{Bre02}.

When $r \leq 1$ and $p>2$, the conjecture was proved by Breuil in  \cite{Bre00} and \cite{Bre02}. In \cite{Bre99a}, Breuil showed that there always exists at least one strongly divisible lattice in the filtered $(\varphi, N)$-modules over $S$ when $er<p-1$. Using this result, Breuil proved the conjecture for general semi-stable representations in the case $e=1$ and $r< p-1$, and Caruso proved the case $er<p-1$ (\cite{Car08}). Later, by utilizing Kisin modules ($\varphi$-modules over $\huaS$) from \cite{Kis06}, Liu (\cite{Liu08}) completely proved part (1) of the conjecture.

In this paper, we will prove part (2), i.e., the $r =p-1$ unipotent case. We use a similar strategy as that of \cite{Liu08}.
Let $\pi$ be a fixed uniformizer of $K$, $\{\pi_n\}$ a fixed system of elements in $\overline{K}$ such that $\pi_1=\pi$, and $\pi_{n+1}^p =\pi_n$ for all $n$. Let $K_n =K(\pi_n)$, $K_{\infty} =\cup_{n \geq 1}K_n$, and $G_{\infty} = \Gal(\overline{K}/K_{\infty})$.
Recall that in \cite{Kis06}, Kisin constructed the category of $\varphi$-modules over $\huaS$ of finite height ($\Mod_{\huaS}^{\varphi}$, see Section 2), and a fully faithful functor $T_{\huaS}$ which sends such $\huaS$-modules to $\Zp$-representations of $G_{\infty}$.
In \cite{Kis06}, Kisin proved that for any $G_{\infty}$-stable $\Zp$-lattice $T$ in a semi-stable Galois representation, there always exists some $\huaM \in \Mod_{\huaS}^{\varphi}$ such that $T_{\huaS}(\huaM)=T$. Now by using Breuil's functor $\M_{\huaS}: \huaM \leadsto S\otimes_{\varphi, \huaS}\huaM$, it will give us a ``quasi-strongly divisible lattice", i.e., a strongly divisible lattice without considering monodromy. A key theorem of our paper (Theorem \ref{mtoM}) is to prove that this functor between $\huaS$-modules and $S$-modules is an equivalence in the unipotent situation. This equivalence and the full faithfulness of $T_{\huaS}$ will lead to the full faithfulness of $T_{\textnormal{st}}$.
Finally, by ``adding in monodromy", i.e., by showing that for any $G_K$-stable $\Zp$-lattice, the quasi-strongly divisible lattice constructed above is in fact a strongly divisible lattice (monodromy stable), this will lead to the essential surjectivity of $T_{\textnormal{st}}$.

The paper proceeds as follows. After a brief review of $p$-adic Hodge theory, we will define categories of $\varphi$-modules over $\huaS$, $\Sigma$ and $S$ (denoted as $\Mod_{\huaS}^{\varphi}$, $\Mod_{\Sigma}^{\varphi}$, $\Mod_{S}^{\varphi}$ respectively) in Section 2. We will define notions of ``unipotency" on these modules, and show that they are compatible.
For a $\varphi$-module $M$ over $\Sigma$, we will construct its maximal unipotent submodule explicitly (Theorem \ref{Mm} and Theorem \ref{SigmaM}). This construction, together with the equivalence between categories $\Mod_{\Sigma}^{\varphi}$ and $\Mod_{S}^{\varphi}$, will give us a ``numerical criterion" for unipotency in the category $\Mod_{S}^{\varphi}$ (Corollary \ref{remark}).
The numerical criterion is crucially used in establishing the equivalence between the category of unipotent $\huaS$-modules and the category of unipotent $S$-modules (Theorem \ref{mtoM}). The proof of Theorem \ref{mtoM} imitates that of Lemma 2.2.2 of \cite{CL09}, but requires substantially more careful analysis.
In Section 3, we introduce strongly divisible lattices, and by using the results as well as the proofs in Section 2, we will show that for unipotent semi-stable Galois representations, all quasi-strongly divisible lattices are unipotent (regarded as $\varphi$-modules over $S$). Then we will be able to state our Main Theorem (Theorem \ref{mainthm}). Section 4 is devoted to the proof of the main theorem. The equivalence between categories of unipotent modules proved in Section 2 will be used to prove the full faithfulness of our functor $T_{\text{st}}$. Note that the proof is much simpler than that in \cite{Liu08}, i.e., we do not need Lemma 3.4.7 in \cite{Liu08}. Finally, in subsection 4.2, we will ``add in" monodromy to show the essential surjectivity of $T_{\text{st}}$. The cases when $p>2$ can be dealt with by using the integral theory of \cite{Liu10} and \cite{Liu12}. For $p=2$, we need a separate strategy from \cite{Bre02}.

\noindent\textbf{Acknowlegement}
The author wants to thank his advisor, Professor Tong Liu, for suggesting this problem, for countless advice and suggestions, and for the continuing encouragement. In particular, he read the drafts several times very carefully, pointed out many mistakes and suggested many improvements. Without him, this work cannot come out. The author also would like to heartily thank Professor Eike Lau for pointing out a mistake in the proof of Theorem \ref{SigmatoS} and Theorem \ref{mtoM}.

\noindent\textbf{Notations}

We fix a uniformizer $\pi$ of $K$, and let $E(u) \in W(k)[u]$ be the minimal polynomial for $\pi$ over $K_0$ of degree $e$. We use $(e_1, \cdots, e_d)^{T}$ to denote the transpose of a row vector. We sometimes use boldface letters ($\bolde, \bolda$, etc.) to denote a column vector. We use notations like $\oplus S \bolde$ to denote the space of $S$-span of vectors in $\bolde$, i.e., if $\bolde=(e_1, \ldots, e_d)^T$, then $\oplus S \bolde =\oplus_{i=1}^d Se_i$. Unfortunately, we sometimes use notations like $\oplus \Sigma \bolde$ where $\Sigma$ is a ring (not the summation sign), but it should be clear from the context. Let $\Mat(?)$ be the ring of matrices with all elements in $?$, $\GL(?)$ the invertible matrices, $Id$ the identity matrix. We use $\gamma_i(x) = \frac{x^i}{i!}$ to denote divided powers. We use $\varphi$ and $N$ to denote Frobenius and monodromy actions on various rings and modules (e.g., $\varphi_{\huaS}, N_{S}$), and when no confusion arises, we omit the subscripts.

\section{Review of $p$-adic Hodge theory}

In this section, we recall some notions in $p$-adic Hodge theory. For more detailed definitions, see e.g., \cite{CF00}.

Let $\overline{K}$ be a fixed algebraic closure of $K$, $\bigO_{\overline{K}}$ the ring of integers, $C_K$ the completion of $\overline{K}$ with respect to the valuation topology.
Recall that $R:=\underleftarrow{\lim}_{n\in\mathbb{N}} \bigO_{\overline{K}}/p\bigO_{\overline{K}}$, where the inverse limit is by taking $p$-th power consecutively.
Let $W(R)$ be the ring of Witt vectors. There is a surjective $K_0$-homomorphism $\theta: W(R)[\frac{1}{p}] \to C_K$ with the kernel a principle ideal generated by an element $\xi=[\varpi]+p$, where $\varpi\in R$ with $\varpi^{(0)}=-p$ and $[\varpi] \in W(R)$ is a Teichm\"{u}ller lift of $\varpi$.
Let $B_{\text{dR}}^{+}:=\underleftarrow{\lim}_{n\in\mathbb{N}}W(R)[1/p]/(\xi)^{n}$, $\BdR:= \Frac\BdR^{+}=\BdR^{+}[1/\xi]$, and there is a filtration on $B_{\text{dR}}$ by $\Fil^i B_{\text{dR}} = B_{\text{dR}}^{+} \xi^i$ for $i \in \Z$.
Let $\Acris:=\{\sum_{n=0}^{\infty}a_{n}\gamma_{n}(\xi), a_{n} \to 0 \: p-\text{adically in }W(R)\} \subset B_{\text{dR}}^{+}.$
Let $\mu_{p^n} \in \overline{K}, n \geq 1$ be a fixed system of primitive roots of unity such that $\mu_p^p=1$ and $\mu_{p^{n+1}}^p =\mu_{p^n}$ for all $n$. Let $\varepsilon = (1, \mu_p, \mu_{p^2}, \ldots) \in R$, $[\varepsilon] \in W(R)$ a Teichm\"{u}ller lift of $\varepsilon$, and $t=\sum_{n=1}^{\infty}(-1)^{n+1}\frac{([\varepsilon]-1)^{n}}{n}\in\BdR^{+}$.
Let $B_{\text{cris}}^{+}:=\Acris[1/p]$, $B_{\text{cris}}:=B_{\text{cris}}^{+}[1/t] \subset B_{\text{dR}}$.
There is a natural Frobenius action $\varphi$ on $B_{\textnormal{cris}}$. Let $B_{\textnormal{st}}:= B_{\text{cris}}[X]$ where $X$ is an indeterminate, there is a natural Frobenius $\varphi$ and monodromy operator $N$ on it. (We omit the precise definition of the Frobenius and monodromy actions.) We have the embeddings $B_{\text{cris}} \subset B_{\text{st}} \subset B_{\text{dR}}$, and we have filtrations on $B_{\text{cris}}$ and $B_{\text{st}}$ induced from that of $B_{\text{dR}}$.

Recall that a $d$-dimensional $p$-adic representation $\rho: G_K \to \GL(V)$ is called
\emph{semi-stable} if $\dim_{K_{0}}(B_{\textnormal{st}}\otimes_{\Qp} V)^{G_{K}}=d$. We denote the category of semi-stable $p$-adic representations as $\Rep_{\Qp}^{\textnormal{st}}(G_K)$.

A \emph{filtered $(\varphi,N)$-module} is a finite dimensional $K_0$-vector space $D$ equipped with
\begin{enumerate}
\item a Frobenius $\varphi: D \to D$, which is semi-linear injective with respect to the arithmetic Frobenius on $K_0$, i.e., $\varphi(ax)=\varphi(a)\varphi(x)$ for all $a \in K_0, x \in D$;
\item a monodromy $N: D \to D$, which is a $K_{0}$-linear map such that $N\varphi=p\varphi N$;
\item a filtration $(\Fil^{i}D_{K})_{i\in\mathbb{Z}}$ on $D_{K}=D\otimes_{K_0} K$, by decreasing $K$-vector subspaces such that $\Fil^{i}D_{K}= D_K$ for $i <<0$ and $\Fil^{i}D_{K}=0$ for $i>>0$.
\end{enumerate}
A morphism between filtered $(\varphi,N)$-modules is a $K_{0}$-linear map which is compatible with $\varphi,N$ and $\Fil^{i}_K$.
We use $\MF$ to denote the category of all filtered $(\varphi,N)$-modules. A sequence $0 \to D_1 \to  D \to D_2 \to 0$ in $\MF$ is called short exact if it is short exact as $K_0$-vector spaces and the sequences on filtrations $0 \to \Fil^i D_{1, K} \to \Fil^i D_{K} \to \Fil^i D_{2, K} \to 0$ are exact for all $i$. In this case, we call $D_2$ a quotient of $D$.

For $D \in \MF$, define:
\begin{itemize}
\item $t_{H}(D)=\sum_{i\in\Z}i\dim_{K}\text{gr}^{i}D_{K}$, where $\text{gr}^{i}D_{K}=\Fil^{i}D_{K}/\Fil^{i+1}D_{K}$.
\item $t_{N}(D)=v_{p}(\det A)$, where $A$ is the matrix for $\varphi$ with respect to some basis $(e_{1},\ldots,e_{d})^T$ of $D$, i.e., $\varphi(e_{1},\ldots,e_{d})^T=A(e_{1},\ldots,e_{d})^T$. Note that $v_{p}(\det A)$ does not depend on the choice of basis.
\end{itemize}

A filtered $(\varphi,N)$-module $D$ is called \emph{weakly admissible} if $t_{H}(D)=t_{N}(D)$
and $t_{H}(D')\leq t_{N}(D')$ for any filtered $(\varphi,N)$-submodule $D'$ of $D$. We use $\MFwa$ to denote the category of weakly admissible filtered $(\varphi,N)$-modules.

\begin{thm}[\cite{CF00}] \label{CF00}
The functor $D_{\textnormal{st}}(V):=(B_{\textnormal{st}}\otimes_{\Qp}V^{\ast})^{G_{K}}$ induces an exact anti-equivalence between $\Rep_{\Qp}^{\textnormal{st}}(G_{K})$ and $\MFwa$, where $V^{\ast}$ is the dual representation of $V$. A quasi-inverse is given by: $$V_{\textnormal{st}}(D) := \Hom_{\varphi, N}(D, B_{\textnormal{st}}) \cap \Hom_{\Fil^{\bullet}}(D_K, K\otimes_{K_0} B_{\textnormal{st}}).$$
\end{thm}

\begin{rem}
The notation here is the same as that in Convention 2.1.1 in \cite{Liu08}, i.e., our $D_{\textnormal{st}}$ is $D_{\textnormal{st}}^{\ast}$ in \cite{Bre02} and \cite{CF00}.
We define the Hodge-Tate weights $\HT(D) = \{i, \textnormal{gr}^{i}D_{K} \neq 0  \}$. In this article, we will always assume that $\HT(D) \subseteq \{ 0, \ldots, r\}$ where $r$ is a nonnegative integer. And by the above theorem, we define the Hodge-Tate weights of a semi-stable representations $V$ as $\HT(V) =\HT(D_{\textnormal{st}}(V))$. Thus $\HT(\chi_p)= {1}$ for the $p$-adic cyclotomic character $\chi_p$.
\end{rem}

For a fixed chosen $r$ such that $\HT(D) \subseteq \{ 0, \ldots, r\}$, the \emph{Cartier dual} of $D \in \MF$ is defined by $D^{\vee}= \Hom_{K_0} (D, K_0)$. Let $\bolde=(e_1, \ldots, e_d)^T$ be a $K_0$-basis of $D$, $\bolde^{\vee}=(e^{\vee}_1, \ldots, e^{\vee}_d)^T$ the dual basis of $D^{\vee}$. Let $A \in \GL_d(K_0), B \in \Mat_d(K_0)$ be the matrices such that $\varphi\bolde=A\bolde, N\bolde=B\bolde$. Then we define $\varphi^{\vee}$ and $N^{\vee}$ by letting $\varphi^{\vee}\bolde^{\vee}=p^r(A^{-1})^T\bolde^{\vee}$ and $N^{\vee}\bolde^{\vee}=-B^T\bolde^{\vee}$. The filtration on $D^{\vee}_K$ is defined by:
$$\Fil^i D^{\vee}_K = (\Fil^{r+1-i} D_K)^{\perp} := \{\ell \in D^{\vee}_K, \Fil^{r+1-i} D_K \subseteq \Ker \ell \}.$$
It is easy to check that $\HT(D^{\vee}) \subseteq \{ 0, \ldots, r\}$, and $D$ is weakly admissible if and only if $D^{\vee}$ is so. When $D$ is weakly admissible, i.e., when $V_{\rm st}(D)$ is semi-stable, we have $D^{\vee} = D_{\rm{st}}( (V_{\rm{st}}(D))^{\ast}\otimes \chi_p^r )$. It is also easy to check that the functor of taking Cartier duals induces a duality on $\MF$ (resp. $\MFwa$), and it is exact.

\begin{defn}
For $D \in \MF$ with Hodge-Tate weights in $\{ 0, \ldots, r\}$,
\begin{enumerate}
\item  $D$ is called \'{e}tale if $\Fil^r D_K =D_K$, it is called multiplicative if $\Fil^1 D_K =\{0\}$.
\item $D$ is called nilpotent if it does not have nonzero multiplicative submodules, it is called unipotent if it does not have nonzero \'{e}tale quotients.
\item For $D \in \MFwa$, the definition of nilpotency and unipotency is similar, except that the multiplicative submodules or \'{e}tale quotients have to lie in $\MFwa$ too.
\end{enumerate}
\end{defn}

\begin{rem}
\begin{enumerate}
\item Both $\MF$ and $\MFwa$ are abelian categories, thus $D$ is multiplicative (resp. nilpotent) if and only if $D^{\vee}$ is \'{e}tale (resp. unipotent), and vice versa.

\item By Theorem \ref{CF00}, given a semi-stable representation $V$, let $D$ be the corresponding weakly-admissible filtered $(\varphi, N)$-module. Then it is easy to show that $D$ is multiplicative if and only if $V$ is an unramified representation, $D$ is nilpotent if and only if $V$ contains no nonzero unramified quotient.
\end{enumerate}
\end{rem}

\section{$\huaS$-modules, $\Sigma$-modules and $S$-modules}

In this section, we study Kisin modules and Breuil modules which are important in integral $p$-adic Hodge theory. In what follows, $r$ is a natural number such that $0 \leq r \leq p-1$.

\subsection{$\huaS$-modules (Kisin modules)}

Recall that $\mathfrak{S}=W(k)[\![u]\!]$ with the Frobenius endomorphism $\varphi_{\huaS}: \huaS \to \huaS$ which acts on $W(k)$ via arithmetic Frobenius and sends $u$ to $u^p$. Denote $\huaS_n = \huaS/p^n\huaS$.

Let $'\ModhuaS$ be the category whose objects are $\huaS$-modules $\huaM$, equipped with $\varphi:\huaM\to\huaM$ which is a
$\varphi_{\huaS}$-semi-linear morphism such that the span of $\text{Im}(\varphi)$ contains $E(u)^{r}\huaM$.
The morphisms in the category are $\huaS$-linear maps that commute with $\varphi$.
Let $\ModFI_{\mathfrak{S}}^{\varphi}$ be the full subcategory of $'\ModhuaS$ with $\huaM \simeq \oplus_{i \in I}\mathfrak{S}_{n_{i}}$ where $I$ is a finite set. Let $\Mod_{\huaS_1}^{\varphi}$ be the full subcategory of $\ModFI_{\mathfrak{S}}^{\varphi}$ with $\huaM \simeq \oplus_{i \in I}\huaS_1$ where $I$ is a finite set.
Let $\ModhuaS$ be the full subcategory of $'\ModhuaS$ with $\huaM$ finite free over $\huaS$.

For $\huaM \in$ $'\ModhuaS$, by $\varphi:\huaS\to\huaS$, we have the $\huaS$-linear map $1\otimes\varphi:\huaS \otimes_{\varphi, \huaS}\huaM \to \huaM$ which sends $s\otimes m$ to $s\varphi(m)$. Let $\varphi^{\ast}\huaM$ denote the image.

We call $\huaM_2$ a quotient of $\huaM$ in the category $\ModhuaS$ if there is a short exact sequence $0 \to \huaM_1 \to \huaM \to \huaM_2 \to 0$ in the category, here short exact means short exact as $\huaS$-modules.

For $\huaM \in \ModhuaS$, the Cartier dual is defined by $\huaM^{\vee}=\Hom_{\huaS}(\huaM,\huaS)$, and $\varphi^{\vee}: \huaM^{\vee} \to \huaM^{\vee}$ is defined such that $\langle \varphi^{\vee}(x^{\vee}) , \varphi(y)\rangle = E(u)^r \langle x^{\vee}, y \rangle, \forall x^{\vee} \in \huaM^{\vee}, y \in \huaM$. It is easy to check that the functor of taking Cartier duals induces a duality on the category $\ModhuaS$, and it transforms short exact sequences to short exact sequences.

\begin{defn}
Let $\huaM \in \ModhuaS$,
\begin{enumerate}
 \item $\huaM$ is called \'{e}tale (resp. multiplicative) if $\varphi^{\ast}\huaM$ is equal to $E(u)^{r}\huaM$ (resp. $\huaM$).
\item $\huaM$ is called nilpotent if it has no nonzero multiplicative submodules, it is called unipotent if it has no nonzero \'{e}tale quotients.
\end{enumerate}
\end{defn}

\begin{lemma} \label{huaMker}
Given a surjective morphism $f: \huaM_1 \twoheadrightarrow \huaM_2$ in $\ModhuaS$, the kernel $\huaM$ is also in $\ModhuaS$.
\end{lemma}
\begin{proof}
We first prove that $\huaM$ is finite free over $\huaS$. Since $\huaM_1/\huaM = \huaM_2$ is $u$-torsion free, $\huaM/u\huaM \hookrightarrow \huaM_1/u\huaM_1$. Thus we have a short exact sequence of $W(k)$-modules: $0 \to \huaM/u\huaM \to \huaM_1/u\huaM_1 \to \huaM_2/u\huaM_2 \to 0$. Since $W(k)$ is a PID, $\huaM/u\huaM$ is finite free. Suppose $(e_1, \ldots, e_{d_1})$ is a basis of $\huaM/u\huaM$, and lift it to $(\hat{e}_1, \ldots, \hat{e}_{d_1})$ in $\huaM$, which generates $\huaM$ by Nakayama Lemma. Take a basis $(\hat{f}_1, \ldots, \hat{f}_{d_2})$ of $\huaM_2$ and lift it to $(\hat{g}_1, \ldots, \hat{g}_{d_2})$ in $\huaM_1$, then clearly the reduction modulo $u$ of $(\hat{e}_1, \ldots, \hat{e}_{d_1}, \hat{g}_1, \ldots, \hat{g}_{d_2})$ is a basis of $\huaM_1/u\huaM_1$, so $(\hat{e}_1, \ldots, \hat{e}_{d_1}, \hat{g}_1, \ldots, \hat{g}_{d_2})$ is in fact a basis of $\huaM_1$. Thus $(\hat{e}_1, \ldots, \hat{e}_{d_1})$ is a basis of $\huaM$, and $\huaM$ is finite free.
$\huaM$ is clearly $\varphi$-stable. To show that the span of $\varphi(\huaM)$ contains $E(u)^r\huaM$, just look at the matrix of $\varphi$ for $\huaM_1$ with respect to the basis $(\hat{e}_1, \ldots, \hat{e}_{d_1}, \hat{g}_1, \ldots, \hat{g}_{d_2})$, which will be block upper-triangular. Span of $\varphi(\huaM_1)$ contains $E(u)^r\huaM_1$ will imply that span of $\varphi(\huaM)$ contains $E(u)^r\huaM$.
\end{proof}

\begin{prop} \label{huaMexact}
Let $\huaM \in \ModhuaS$,
\begin{enumerate}
\item   $\huaM$ is \'{e}tale (resp. unipotent) if and only if $\Mdual$ is multiplicative (resp. nilpotent), and vice versa.

\item Let $(\varphi^{*})^{n}\huaM$ be the $\huaS$-submodule of $\huaM$ generated by $\varphi^{n}(\huaM)$.
Then $\huaM^{\textnormal{m}}:=\cap_{n=1}^{\infty}(\varphi^{*})^{n}\huaM$ is in $\Mod_{\mathfrak{S}}^{\varphi}$, and it is the maximal multiplicative submodule of $\huaM$.
The quotient $\huaM/\huaM^{\textnormal{m}}$ is in $\Mod_{\mathfrak{S}}^{\varphi}$ as well.

\item   $\huaM$ is nilpotent if and only if $(\varphi^{*})^{n}\huaM\subseteq(p,u)\huaM$ for $n \gg 0$.

\item We have short exact sequences
$$ 0\to\huaM^{\textnormal{m}}\to\huaM\to\huaM^{\textnormal{nil}}\to 0$$
and
$$ 0 \to \huaM^{\textnormal{uni}} \to \huaM \to \huaM^{\textnormal{et}} \to 0,$$
where $\huaM^{\textnormal{m}}, \huaM^{\textnormal{nil}}, \huaM^{\textnormal{uni}}, \huaM^{\textnormal{et}}$ are maximal multiplicative submodule, maximal nilpotent quotient, maximal unipotent submodule, maximal \'{e}tale quotient of $\huaM$ in the category $\ModhuaS$ respectively (by maximal quotient, we mean any other quotient is a quotient of our maximal quotient).

\end{enumerate}
\end{prop}

\begin{proof}
For (1), (2), (3), they are proved in Proposition 1.2.11 in \cite{Kis09a}, and Lemma 1.2.2 in \cite{Kis09b}. Remark that suppose the matrix of $\varphi$ with respect to a basis of $\huaM$ is $A$, then there exists a matrix $A'$ with elements in $\huaS$ such that $AA' = E(u)^rId$. $\huaM$ is nilpotent if and only if $\varphi^{N}(A)\cdots \varphi(A)A$ converges to $0$ as $N \to \infty$, and it is unipotent if and only if $\Pi_{n=0}^{\infty} \varphi^{n}(A')=0$ (since $(A')^T$ is the matrix of $\varphi^{\vee}$ for $\huaM^{\vee}$ with respect to the dual basis ).

For (4), the maximality of $\huaM^{\textnormal{m}}$ is evident from the definition. $\huaM^{\textnormal{nil}}$ is defined as $\huaM/\huaM^{\textnormal{m}}$, it is easily shown to be nilpotent. To show maximality, suppose $\huaM_1$ is another nilpotent quotient defined by $0 \to \huaM_2 \to \huaM \to \huaM_1 \to 0$. We claim that the composite $\huaM^{\textnormal{m}} \to \huaM \to \huaM_1$ is zero map. Since otherwise, the nonzero image will be contained in $\huaM_1^{\textnormal{m}}$, contradicting that $\huaM_1$ is nilpotent. Since $\huaM^{\textnormal{nil}}$ is the cokernel of $\huaM^{\textnormal{m}} \to \huaM$, there is a surjective map $\huaM^{\textnormal{nil}} \to \huaM_1$. The map is clearly compatible with $\varphi$-structures, so it is a surjective morphism in the category $\Mod_{\huaS}^{\varphi}$. The kernel of this surjective morphism is also a module in $\Mod_{\huaS}^{\varphi}$ by Lemma \ref{huaMker}. Thus, $\huaM_1$ is a quotient of $\huaM^{\textnormal{nil}}$.

For the second short exact sequence, it is defined by taking the Cartier dual of the first sequence, i.e., $\huaM^{\textnormal{et}}= ((\huaM^{\vee})^{\textnormal{m}})^{\vee}$ and $\huaM^{\textnormal{uni}}= ((\huaM^{\vee})^{\textnormal{nil}})^{\vee}$.

To show the maximality of $\huaM^{\textnormal{et}}$, suppose $\huaM_1$ is an \'{e}tale quotient by $0 \to \huaM_2 \to \huaM \to \huaM_1 \to 0$. Then the composite $\huaM^{\textnormal{uni}} \to \huaM \to \huaM_1$ is zero map since the dual morphism $(\huaM_1)^{\vee} \to (\huaM^{\textnormal{uni}})^{\vee}$ has to be zero. Since $\huaM^{\textnormal{et}}$ is the cokernel of $\huaM^{\textnormal{uni}} \to \huaM$, there is a surjective morphism $\huaM^{\textnormal{et}} \to \huaM_1$. Now Lemma \ref{huaMker} implies that $\huaM^{\textnormal{et}}$ is maximal.

To show the maximality of $\huaM^{\textnormal{uni}}$, take any unipotent submodule $\huaM_1$ of $\huaM$. Then the composite $\huaM_1 \hookrightarrow \huaM \to \huaM^{\textnormal{et}}$ is zero map since the dual morphism $(\huaM^{\textnormal{et}})^{\vee} \to \huaM_1^{\vee}$ has to be zero map. Since $\huaM^{\textnormal{uni}}$ is the kernel of $\huaM \to \huaM^{\textnormal{et}}$, $\huaM_1$ injects into $\huaM^{\textnormal{uni}}$.
\end{proof}

Embed $\huaS \hookrightarrow W(R)$ by $u \mapsto [\underline{\pi}]$, where $[\underline{\pi}]$ is a Teichm\"{u}ller lift of $\underline{\pi} = (\pi, \pi_2, \ldots)$. Let $\bigOE =p$-adic completion of $\huaS[\frac{1}{u}]$, $\mathcal{E}$ the fraction field of $\bigOE$, $\mathcal{E}^{\text{ur}}$ the maximal unramified extension of $\mathcal{E}$ in $W(R)[\frac{1}{p}]$, $\widehat{\mathcal{E}^{\text{ur}}}$ the $p$-adic completion of $\mathcal{E}^{\text{ur}}$, and $\huaS^{\text{ur}} = \mathcal{O}_{\widehat{\mathcal{E}^{\text{ur}}}} \cap W(R)$.

We note that $G_{\infty}$ acts naturally on $\huaS^{\text{ur}}$ and $\mathcal{O}_{\widehat{\mathcal{E}^{\text{ur}}}}$, and fixes $\huaS$. By \cite{Fon90}, for $\huaM \in \ModFI_{\huaS}^{\varphi}$, let
$$ T_{\huaS}(\huaM) : =  \Hom_{\huaS, \varphi} (\huaM, \huaS^{\textnormal{ur}}[1/p]/\huaS^{\textnormal{ur}}),$$
it is a finite torsion $\Zp$-representation of $G_{\infty}$.
For $\huaM \in \Mod_{\huaS}^{\varphi}$ which is finite free of rank $d$, let
$$  T_{\huaS}(\huaM)  :=  \Hom_{\huaS, \varphi} (\huaM, \huaS^{\textnormal{ur}}) ,$$
it is a finite free $\Zp$-representation of $G_{\infty}$ of rank $d$.

\subsection{$S$-modules (Breuil modules)}

Recall that $S$ is the $p$-adic completion of the PD-envelope of $W(k)[u]$ with respect to the ideal $(E(u))$. It is a $W(k)[u]$-subalgebra of $K_{0}[\![u]\!]$, and $S = \{\sum_{i=0}^{\infty} a_i \frac{E(u)^i}{i!} | a_i \in W(k)[u], a_i \to 0 \ p-\text{adically}  \}$.
$S$ has a filtration $\{\Fil^{j}S\}_{j \geq 0}$, where $\Fil^{j}S$ is the $p$-adic completion of the ideal generated by all $\gamma_{i}(E(u))=\frac{E(u)^{i}}{i!}$ with $i\geq j$. Note that $S/\Fil^{j}S=W(k)[u]/(E(u)^{j})$ for $j \leq p$.
There is a Frobenius $\varphi: S \to S$ which acts on $W(k)$ via arithmetic Frobenius and sends $u$ to $u^p$, and there is a $W(k)$-linear differential operator $N$ (called the monodromy operator) such that $N(u) = -u$.
We have $\varphi(\Fil^j S) \subset p^jS$ for $1 \leq j \leq p-1$, and we denote $\varphi_j = \frac{\varphi}{p^j}: \Fil^j S \to S$.
We denote $c =\frac{\varphi(E(u))}{p}$ which is a unit in $S$, and $S_n= S/p^nS$.

Let $'\Mod_{S}^{\varphi}$ be the category whose objects are triples $(\M, \Fil^r \M, \varphi_r)$ where
$\bigM$ is an $S$-module,
$\Fil^r \M \subseteq \M$ is an $S$-submodule which contains $\Fil^r S \cdot \M$,
and $\varphi_r : \Fil^r \M \to \M$ is a $\varphi$-semi-linear map such that $\varphi_r(sx)=c^{-r}\varphi_r(s)\varphi_r(E(u)^rx)$ for $s \in \Fil^r S$ and $x \in \M$.
The morphisms in the category are $S$-linear maps preserving $\Fil^r$ and commuting with $\varphi_r$.
Let $\ModFI_{S}^{\varphi}$ be the full subcategory of  $'\Mod_{S}^{\varphi}$ with $\M \simeq \oplus_{i \in I} S_{n_i}$ with $I$ a finite set and $\varphi_r(\Fil^r \M)$ generates $\M$. Let $\Mod_{S_1}^{\varphi}$ be the full subcategory of $\ModFI_{S}^{\varphi}$ such that $\M \simeq \oplus_{i \in I} S_1$ with $I$ a finite set.
Let $\Mod_{S}^{\varphi}$ be the full subcategory of $'\Mod_{S}^{\varphi}$ such that $\M$ is a finite free $S$-module, $\varphi_r(\Fil^r \M)$ generates $\M$, and $\M/\Fil^r \M$ is $p$-torsion free. Note that the last condition implies that $(\M/p^n\M, \Fil^r \M/p^n\Fil^r \M, \varphi_r) \in \ModFI_{S}^{\varphi}$ for all $n$.

\begin{rem} \label{phiphir}
Note that for $\M \in \Mod_{S}^{\varphi}$, $\varphi_r$ induces a map, $\varphi: \M \to \M$, $\varphi(x) = \frac{\varphi_r(E(u)^r x)}{c^r}$ (so $\varphi_r =\frac{\varphi}{p^r}$). It is easy to check that one can change the triple $(\M, \Fil^r \M, \varphi_r)$ in the definition of $\Mod_{S}^{\varphi}$ to a new triple $(\M, \Fil^r \M, \varphi)$ where $\varphi: \M \to \M$ such that $\varphi(\Fil^r\M) \subseteq p^r\M$ and $\varphi(\Fil^r\M)$ generates $p^r\M$. I.e., $\varphi_r$ and $\varphi$ provide equivalent information. In the following, we will freely use both of them.
\end{rem}

It is easy to check that $A_{\text{cris}}$  is an object in $'\Mod_{S}^{\varphi}$. Thus by \cite{Bre99a}, for $\M \in \ModFI_{S}^{\varphi}$, let
$$T_{\text{cris}}(\M) := \Hom_{'\Mod_{S}^{\varphi}} (\M, A_{\text{cris}}[1/p]/A_{\text{cris}}),$$
it is a finite torsion $\Zp$-representation of $G_{\infty}$.
And for $\M \in \Mod_{S}^{\varphi}$ which is finite free of rank $d$, let
$$T_{\text{cris}}(\M) := \Hom_{'\Mod_{S}^{\varphi}} (\M, A_{\text{cris}}),$$
it is a finite free $\Zp$-representation of $G_{\infty}$ of rank $d$.

A sequence $0 \to \M_1 \to \M \to \M_2 \to 0$ in $\Mod_{S}^{\varphi}$ is called short exact if it is short exact as a sequence of $S$-modules, and
the sequence on filtraions  $0 \to \Fil^r\M_1 \to \Fil^r\M \to \Fil^r\M_2 \to 0$ is also short exact. In this case, we call $\M_2$ a quotient of
$\M$.

For $\M \in \Mod_{S}^{\varphi}$, the Cartier dual of $\M$ is defined by $\M^{\vee}:= \Hom_S(\M, S)$,
$$\Fil^r \M^{\vee} := \{ f \in \M^{\vee}, f(\Fil^r \M) \subseteq \Fil^r S \},$$
and
$$\varphi_r^{\vee} : \Fil^r \M^{\vee} \to \M^{\vee}, \varphi_r^{\vee}(f)(\varphi_r(x)) = \varphi_r(f(x)), \forall f \in \Fil^r \M^{\vee}, x \in \Fil^r \M$$.

Note that $\varphi_r^{\vee}(f)$ is well defined since $\varphi_r(\Fil^r \M)$ generates $\M$.

By Proposition V3.3.1 of \cite{Car05}, the functor $\M \to \M^{\vee}$ induces a duality on the category $\Mod_{S}^{\varphi}$, and it transforms short exact sequences to short exact sequences.

\begin{defn} \label{defunip}
For $\M \in \Mod_{S}^{\varphi}$,
\begin{enumerate}
\item  $\M$ is called \'{e}tale if $\Fil^r \M =\M$, it is called multiplicative if $\Fil^r \M =\Fil^rS \M$.
\item $\M$ is called nilpotent if it has no nonzero multiplicative submodules, it is called unipotent if it has no nonzero \'{e}tale quotients.
\end{enumerate}

\end{defn}

We record a useful lemma here.

\begin{lemma}[\cite{Liu08}] \label{Filr}
For $\M \in \Mod_S^{\varphi}$ (resp. $\Mod_{S_1}^{\varphi}$), there exists $\bolda = (\alpha_1, \ldots, \alpha_d)^{T}$ with $\alpha_i \in \Fil^r \M$, such that
\begin{enumerate}
\item $\Fil^r \M = \oplus_{i=1}^{d} S \alpha_i + \Fil^p S \M$,
\item $\oplus_{i=1}^{d} S \alpha_i \supseteq E(u)^r \M$, and $\varphi_r(\bolda)$ is a basis of $\M$.
\end{enumerate}
\end{lemma}

\begin{proof}
Let $\M \in \Mod_S^{\varphi}$, then in Proposition 4.1.2 of \cite{Liu08}, the lemma was proved for $r < p-1$.
In fact , the proof still works for $r =p-1$, the only difference is the last sentence. Before the last sentence, we have already proved that there exists $\bolda$, such that $\Fil^r \M =\oplus_{i=1}^{d} S \alpha_i + \Fil^p S \M$ and $\oplus_{i=1}^{d} S \alpha_i \supseteq E(u)^r \M$, so we only need to show that $\varphi_r(\bolda)$ is a basis of $\M$. In the proof of \cite{Liu08}, since $r< p-1$, we have $p \mid \varphi_r(\Fil^p S)$, but for $r=p-1$, this is invalid. But because we have that $\varphi_r(\Fil^r \M)$ generates $\M$, $\varphi_r(\bolda)$ together with $\varphi_r(\Fil^pS \M)$ generates $\M$. Since $\varphi_r(sm)=\varphi_r(s) \frac{\varphi_r(E(u)^rm)}{c^r}$ for $s \in \Fil^pS$, $m \in \M$, and $\varphi_r(E(u)^rx)$ is a linear combination of $\varphi_r(\bolda)$ (because $\oplus_{i=1}^{d} S \alpha_i \supseteq E(u)^r \M$), thus $\varphi_r(\bolda)$ alone generates $\M$, and we conclude the proof for the $r=p-1$ case.

For Let $\M \in \Mod_{S_1}^{\varphi}$, just apply Proposition 2.2.1.3 of \cite{Bre99a} (which is indeed valid for any $r \leq p-1$) and use the above argument.
\end{proof}

\begin{lemma} \label{Smult}
$\M \in \Mod_S^{\varphi}$ is \'{e}tale (resp. multiplicative) if and only if $1\otimes \varphi: S \otimes_{\varphi, S} \M \to \M$, or equivalently, $1\otimes \varphi_r: S\otimes_{\varphi, S} E(u)^r\M \to \M$, is an isomorphism onto $p^r \M$ (resp. $\M$).
\end{lemma}

\begin{proof}
Necessity is clear. To prove sufficiency, denote the image of the map $1\otimes \varphi$ as $\varphi^{\ast}\M$, and suppose $\varphi^{\ast}\M= p^r \M$. As in Lemma \ref{Filr}, suppose $\bolda =A \bolde$, with $\bolde =\varphi_r(\bolda)$ a basis, and $A \in \Mat(S)$. Since $\varphi(\M)$ generates $p^r\M$, $\varphi_r(E(u)^r \bolde) = \varphi_r(A'A\bolde) =\varphi(A') \bolde$ generates $p^r \M$ (here $A'$ is the matrix such that $AA' =E(u)^rId$). That means $\varphi(A')$ is $p^r$ times an invertible matrix. Since $\varphi(A)\varphi(A')=\varphi(E(u)^rId)=p^rc^rId$, $\varphi(A)$ is an invertible matrix. We claim that $A$ is also an invertible matrix, which guarantees that $\Fil^r \M =\M$.
To prove the claim, we need to show that $\det(A)$ is invertible. But in the ring $S$, an element $s$ is invertible if and only if $\varphi(s)$ is, and $\varphi(\det(A)) = \det (\varphi(A))$ is invertible, thus our claim is true. The proof for the multiplicative case is similar.
\end{proof}

\subsection{$\Sigma$-modules}

Recall (cf. \cite{Bre99a} 3.2.1) that the ring $\Sigma = W(k)[\![X]\!][u]/(u^{ep}-pX) = W(k)[\![Y]\!][u]/(E(u)^p -pY)$,
and we have an injection $\Sigma \hookrightarrow S$ by sending $X$ to $u^{ep}/p$ (or $Y$ to $E(u)^p/p$) and $u$ to $u$.
Via this injection, $\Sigma$ is stable under the Frobenius and monodromy on $S$. $\Sigma$ is equipped with the induced filtration $\Fil^i \Sigma=\Sigma \cap \Fil^i S$.
For $0 \leq i \leq p-1$, we have $\Fil^i \Sigma =(E(u)^i, Y)$ and $\varphi(\Fil^i \Sigma) \subseteq p^i\Sigma$.
$\Sigma$ is a Noetherian local ring with maximal ideal $(p, u, X)$ and it is complete with respect to the $(p, u, X)$-adic topology. Note that $c =\frac{\varphi(E(u))}{p}$ is also a unit in $\Sigma$, and denote $\Sigma_n =\Sigma/p^n\Sigma$.

Let $'\Mod_{\Sigma}^{\varphi}$ be the category whose objects are triples $(M, \Fil^rM, \varphi_r)$, where $M$ is a $\Sigma$-module,
$\Fil^r M \subseteq M$ is a $\Sigma$-submodule which contains $\Fil^r\Sigma\cdot M$, and $\varphi_r: \Fil^r M \to M$ is a Frobenius-semi-linear map such that $\varphi_r(sx) =c^{-r}\varphi_r(s)\varphi_r(E(u)^rx)$ for $s\in \Fil^r \Sigma$
and $x \in M$. We can define the categories $\ModFI_{\Sigma}^{\varphi}$, $\Mod_{\Sigma_1}^{\varphi}$ and $\Mod_{\Sigma}^{\varphi}$ similarly as $\ModFI_{S}^{\varphi}$, $\Mod_{S_1}^{\varphi}$ and $\Mod_{S}^{\varphi}$ in the previous subsection. Morphisms, Cartier duals, short exact
sequences and quotients in the category $\Mod_{\Sigma}^{\varphi}$ are defined analogously. Following a similar proof as that of Proposition V3.3.1 of \cite{Car05}, the Cartier dual functor also induces a duality on the category $\Mod_{\Sigma}^{\varphi}$, and it transforms short exact sequences to short exact sequences. Note that we can also change the $\varphi_r$ in $\Mod_{\Sigma}^{\varphi}$ to $\varphi$ in the same way as in Remark
\ref{phiphir}. For $M \in \Mod_{\Sigma}^{\varphi}$, let $\varphi^{\ast}M$ be the $\Sigma$-span of $\varphi(M)$, and $(\varphi^{\ast})^nM$ the $\Sigma$-span of $\varphi^n(M)$.

\begin{thm} \label{Mm}
For $M \in \Mod_{\Sigma}^{\varphi}$, the submodule $M^{\textnormal{m}}: = \cap_{n=1}^{\infty} (\varphi^{\ast})^{n}M$ is a finite free $\Sigma$-module.
\end{thm}

\begin{lemma} \label{decomposition}
For a finite free $W(k)$-module $M$ with a Frobenius-semi-linear $\varphi$-action, we have a decomposition $M= M_{\textnormal{unit}} \oplus M_{\textnormal{nil}}$, where $\varphi$ is bijective on the first part, and topologically nilpotent on the second part, namely $\varphi^k(M_{\textnormal{nil}}) \to 0$ as $k \to \infty$.
\end{lemma}

\begin{proof}

Note that for any $n$, $W(k)/p^n$ is Artinian and the Frobenius action is bijective, thus by Fitting Lemma, we have a decomposition $M/p^n = (M/p^n)_{\text{unit}} \oplus (M/p^n)_{\text{nil}}$, where $\varphi$ is bijective on the first part and nilpotent on the second part.

Now we show that these decompositions are compatible with each other, namely $(M/p^n)_{\text{nil}} \otimes W(k)/p^{n-1} \simeq (M/p^{n-1})_{\text{nil}}$, and $(M/p^n)_{\text{unit}} \otimes W(k)/p^{n-1} \simeq (M/p^{n-1})_{\text{unit}}$.

It suffices to prove the first one. Injectivity is clear because the map is induced from the isomorphism $M/p^n\otimes W(k)/p^{n-1} \simeq M/p^{n-1}$. Surjectivity: suppose $x \in (M/p^{n-1})_{\text{nil}}$, and take any lift $\hat{x} \in M/p^n$, then $\varphi^k(\hat{x}) = p^{n-1}y$ for $k$ large enough and here $y \in M/p^n$. Decompose $y =y_1+y_2$ where $y_1 \in (M/p^n)_{\textnormal{unit}}, y_2 \in (M/p^n)_{\textnormal{nil}}$, then $\varphi^N (\hat{x}) =p^{n-1}\varphi^{N-k} (y_1) = p^{n-1}\varphi^N(y_4)$ for $N$ large enough and some $y_4 \in (M/p^n)_{\textnormal{unit}}$, then we have $\varphi^N (\hat{x} -p^{n-1}y_4) =0$, thus $\hat{x} - p^{n-1}y_4 \in (M/p^n)_{\text{nil}}$, and it maps to $x$.

Now for a finite free $M$, let $M_{\text{nil}} =: \varprojlim (M/p^n)_{\text{nil}}$, and
$M_{\text{unit}} =:\varprojlim (M/p^n)_{\text{unit}}$. Then apparently $M_{\text{nil}} \cap M_{\text{unit}} =0$. For any $x \in M$, we have unique decompositions $x \bmod p^n = y_n + z_n$ which are compatible with each other, take inverse limit we have that $M = M_{\text{nil}} + M_{\text{unit}}$, thus $M = M_{\text{unit}} \oplus M_{\text{nil}}$.

\end{proof}

\begin{proof} \textbf{Proof of Theorem \ref{Mm}}: Apply Lemma \ref{decomposition} to $M/(u, X)M$, then \\$M/(u ,X)M \simeq (M/(u ,X)M)_{\text{unit}} \oplus (M/(u ,X)M)_{\text{nil}}$, where $\varphi$ is bijective on the first part, and topologically nilpotent on the second. We naturally have a map \\$M^{\textnormal{m}}/(u, X)M^{\textnormal{m}} \to (M/(u, X))_{\text{unit}}$, and we claim that it is an isomorphism. To show injectivity, first note that the map $\varphi: M^{\textnormal{m}}/(u, X)M^{\textnormal{m}} \to M^{\textnormal{m}}/(u, X)M^{\textnormal{m}}$ is surjective, thus bijective since $W(k)$ is Noetherian and $M^{\textnormal{m}}/(u, X)M^{\textnormal{m}}$ is finitely generated. Now if $x \in M^{\textnormal{m}}$ and $x \in (u, X)M$, then apparently $\varphi^k(x)$ converges to $0$ in $M$ (by the $(p, u, X)$-adic topology), thus also in $M^{\textnormal{m}}$ by Artin-Rees Lemma. But because of the isomorphism $\varphi: M^{\textnormal{m}}/(u, X)M^{\textnormal{m}} \to M^{\textnormal{m}}/(u, X)M^{\textnormal{m}}$, $\varphi^k(x)$ converges to $0$ if and only if $x \in (u, X)M^{\textnormal{m}}$, thus we have shown the injectivity. To show surjectivity, take an element $x_0 \in (M/(u, X))_{\text{unit}}$, and suppose we have $x_0 = \varphi^n(x_n)$ with $x_n \in (M/(u, X))_{\text{unit}}$, then lift all the $x_n$ to some $\hat{x}_n \in M$, we claim that $\varphi^n(\hat{x}_n)$ converges to an element in $M^{\textnormal{m}}$ and maps to $x_0$. To prove the claim, take any $n>m$, then $\varphi^{n-m}(\hat{x}_n) - \hat{x}_m \equiv \varphi^{n-m}(x_n)-x_m \equiv 0 (\bmod (u, X)M)$, thus $\varphi^n(\hat{x}_n) - \varphi^m(\hat{x}_m) \in \varphi^{m}((u, X)M)$. So $\{\varphi^n(\hat{x}_n)\}_{n \geq 0}$ clearly converges to an element $y \in M$. Similarly $\{\varphi^{n-1}(\hat{x}_n)\}_{n \geq 1}$ converges to some $y_1 \in M$, and $\varphi(y_1) =y$. Similarly, we can find $y_n \in M$ such that $y =\varphi^n(y_n)$, thus $y \in \cap_{n=1}^{\infty} (\varphi^{\ast})^n M =M^{\textnormal{m}}$, and clearly $y$ maps to $x_0$.

Now we have proved that $M^{\textnormal{m}}/(u, X)M^{\textnormal{m}}$ is a finite free $W(k)$-module, i.e., \\
$(\Sigma/(u, X))^{d_1} \simeq M^{\textnormal{m}}/(u, X)M^{\textnormal{m}}$. Take a basis $(e_1, \ldots, e_{d_1})$ of $M^{\textnormal{m}}/(u, X)M^{\textnormal{m}}$, and lift it to $(\hat{e}_1, \ldots, \hat{e}_{d_1})$ in $M^{\textnormal{m}}$, since $(u, X)$ is in the maximal ideal of $\Sigma$, by Nakayama Lemma, the span of $(\hat{e}_1, \ldots, \hat{e}_{d_1})$ generates $M^{\textnormal{m}}$. Also we take a basis $(f_1, \ldots, f_{d_2})$ of $(M/(u ,X))_{\text{nil}}$ and lift it to $(\hat{f}_1, \ldots, \hat{f}_{d_2})$ in $M$. Now since the reduction of $(\hat{e}_1, \ldots, \hat{e}_{d_1}, \hat{f}_1, \ldots, \hat{f}_{d_2})$ modulo $(u, X)$ is a basis of $M/(u, X)M$, $(\hat{e}_1, \ldots, \hat{e}_{d_1}, \hat{f}_1, \ldots, \hat{f}_{d_2})$ is clearly a basis of $M$. Thus $(\hat{e}_1, \ldots, \hat{e}_{d_1})$ is also a basis of $M^{\textnormal{m}}$, and this proves that $M^{\textnormal{m}}$ is finite free.

\end{proof}

\begin{defn}
For $M \in \Mod_\Sigma^{\varphi}$,
\begin{enumerate}
\item  $M$ is called \'{e}tale if $\Fil^r M =M$, it is called multiplicative if $\Fil^r M =\Fil^r\Sigma M$.
\item $M$ is called nilpotent if it has no nonzero multiplicative submodules, it is called unipotent if it has no nonzero \'{e}tale quotients.
\end{enumerate}
\end{defn}

\begin{lemma} \label{SigmaFilr}
For $M \in \Mod_\Sigma^{\varphi}$ (resp. $\Mod_{\Sigma_1}^{\varphi}$), there exists $\bolda = (\alpha_1, \ldots, \alpha_d)^T$ with $\alpha_i \in \Fil^r M$, such that
\begin{enumerate}
\item $\Fil^r M = \oplus_{i=1}^{d} \Sigma \alpha_i + \Fil^p \Sigma M$,
\item $\oplus_{i=1}^{d} \Sigma \alpha_i \supseteq E(u)^r M$, and $\varphi_r(\bolda)$ is a basis of $M$.
\end{enumerate}
\end{lemma}

\begin{proof}
This is the $\Sigma$-module version of Lemma \ref{Filr}, and the proof is similar.
\end{proof}

\begin{lemma} \label{Sigmamult}
$M \in \Mod_\Sigma^{\varphi}$ is \'{e}tale (resp. multiplicative) if and only if $1\otimes \varphi: \Sigma \otimes_{\varphi, \Sigma} M \to M$, or
equivalently, $1\otimes \varphi_r: \Sigma \otimes_{\varphi, \Sigma} E(u)^rM \to M$, is an isomorphism onto $p^r M$ (resp. $M$).
\end{lemma}

\begin{proof}
This is the $\Sigma$-module version of Lemma \ref{Smult}.
\end{proof}

\begin{lemma} \label{pMm}
 Let $M \in \Mod_\Sigma^{\varphi}$ be a multiplicative module, $\varphi: M \to M$ the morphism derived from $\varphi_r$. If $\varphi(z) \in pM$, then $z \in (p, E(u), \frac{E(u)^p}{p})M$.
\end{lemma}

\begin{proof}
 Take any basis $\bolde =(e_1, \ldots, e_d)^T$ of $M$, then since $M$ is multiplicative, $\varphi(\bolde) = A\bolde$ with $A$ an invertible matrix (Lemma \ref{Sigmamult}). Suppose $z =\sum_{i=1}^{d} a_i e_i$ and $\varphi(z) \in pM$, then $\varphi(z) =(\varphi(a_1), \ldots, \varphi(a_d)) A \bolde = p(b_1, \ldots, b_d)\bolde$ for some $b_i \in \Sigma, 1 \leq i \leq d$.
 Thus $(\varphi(a_1), \ldots, \varphi(a_d)) = p(b_1, \ldots, b_d)A^{-1}$, so $\varphi(a_i) \in p\Sigma$ for all $i$. We can easily check that $\varphi(s) \in p\Sigma$ implies $s \in (p, E(u), \frac{E(u)^p}{p})\Sigma$. Thus $a_i \in (p, E(u), \frac{E(u)^p}{p})\Sigma$ for all $i$, and $z \in (p, E(u), \frac{E(u)^p}{p})M$.
\end{proof}

\begin{thm} \label{SigmaM}
For $M \in \Mod_{\Sigma}^{\varphi}$,
\begin{enumerate}
\item We have two short exact sequences in $\Mod_{\Sigma}^{\varphi}$:
$$ 0\to M^{\textnormal{m}}\to M\to M^{\textnormal{nil}} \to 0 $$
and
$$ 0 \to M^{\textnormal{uni}} \to M \to M^{\textnormal{et}} \to 0,$$
where $M^{\textnormal{m}}$ (resp. $M^{\textnormal{nil}}, M^{\textnormal{uni}}, M^{\textnormal{et}}$) is a multiplicative (resp. nilpotent, unipotent, \'{e}tale) module. In fact, the second sequence is by taking Cartier dual of the first sequence, i.e., $M^{\textnormal{uni}}=(M^{\vee, \textnormal{nil}})^{\vee}$ and $M^{\textnormal{et}}=(M^{\vee, \textnormal{m}})^{\vee}$. Also, $M^{\textnormal{m}}$ is the maximal multiplicative submodule of $M$.

\item $M$ is \'{e}tale (resp. unipotent) if and only if $M^{\vee}$ is multiplicative (resp. nilpotent), and vice versa.

\end{enumerate}
\end{thm}

\begin{proof}
Define $M^{\textnormal{m}} = \cap_{n=1}^{\infty} (\varphi^{\ast})^{n}M$ as in Theorem \ref{Mm}. Define $M^{\textnormal{nil}}
=M/M^{\textnormal{m}}$. We claim that $M^{\textnormal{nil}}$ is a finite free $\Sigma$-module. To prove the claim, note that in the proof of Theorem \ref{Mm}, we have shown that $M^m/(u, X) \simeq (M/(u, X))_{\textnormal{unit}}$, in particular, $M^m/(u, X)$ injects into $M/(u, X)$. Thus, the short exact sequence $0 \to M^m \to M \to M^{\textnormal{nil}} \to 0$ is still short exact after reduction modulo $(u, X)$, so $M^{\textnormal{nil}}/(u, X) \simeq
(M/(u, X))_{\textnormal{nil}}$ is a finite free $W(k)$-module (Lemma \ref{decomposition}), suppose it is of rank $d_2$. By Nakayama Lemma, we can have a surjective map $(\Sigma)^{d_2} \twoheadrightarrow M^{\text{nil}}$, and $M^{\text{nil}}\otimes_{\Sigma} \Frac \Sigma \simeq (M\otimes\Frac\Sigma) / (M^{\textnormal{m}} \otimes \Frac\Sigma)\simeq (\Frac\Sigma)^{d_2} $.
Thus we have the following commutative diagram,

$$\xymatrix{
\Sigma^{d_2}   \ar@{->>}[d]   \ar@{^{(}->}[r]   &(\Frac \Sigma)^{d_2}  \ar[d]^{=} \\
M^{\text{nil}}  \ar[r]                             &M^{\text{nil}}\otimes \Frac \Sigma
}$$

 So our surjective map $(\Sigma)^{d_2} \twoheadrightarrow M^{\text{nil}}$ has to be injective, thus bijective, i.e., $M^{\textnormal{nil}}$ is finite free.

Now we prove that $M^{\textnormal{m}} \in \Mod_{\Sigma}^{\varphi}$. Since $M/M^{\textnormal{m}}$ is finite free, so in particular $p$-torsion free, $M^{\textnormal{m}} \cap p M = p M^{\textnormal{m}}$. Define $\Fil^r M^{\textnormal{m}} = M^{\textnormal{m}} \cap \Fil^r M$, thus $\varphi(\Fil^r M^{\textnormal{m}}) \subseteq M^{\textnormal{m}} \cap p^r M = p^r M^{\textnormal{m}}$. Since $\Fil^r M^{\textnormal{m}} \supseteq \Fil^r \Sigma M^{\textnormal{m}}$ and $\varphi( \Fil^r \Sigma M^{\textnormal{m}})$ generates $p^r M^{\textnormal{m}}$, thus $\varphi(\Fil^r M^{\textnormal{m}})$ generates $p^rM^{\textnormal{m}}$. Now we claim that $M^{\textnormal{m}}/\Fil^r M^{\textnormal{m}}$ is $p$-torsion free. To prove the claim, suppose $x \in M^{\textnormal{m}}$ and $px \in \Fil^r M^{\textnormal{m}}$, thus $px \in \Fil^r M$. Since $M/\Fil^r M$ is $p$-torsion free, thus $x \in \Fil^r M$, so $x \in \Fil^r M^{\textnormal{m}}$. All these facts show that $(M^{\textnormal{m}}, \Fil^r M^{\textnormal{m}}, \varphi_r)$ is a well-defined object in $\Mod_{\Sigma}^{\varphi}$. And by Lemma \ref{Sigmamult}, it is a multiplicative object.

 Next, we prove that $ M^{\textnormal{nil}} \in \Mod_{\Sigma}^{\varphi}$. Define $\Fil^r M^{\textnormal{nil}} := \Fil^r M/\Fil^r M^{\textnormal{m}}= (\Fil^r M +M^{\textnormal{m}})/M^{\textnormal{m}}$. Clearly $\varphi_r(\Fil^r M^{\textnormal{nil}})$ generates $M^{\textnormal{nil}}$. We want to show that $M^{\textnormal{nil}}/\Fil^r M^{\textnormal{nil}} = (M/M^{\textnormal{m}})/((\Fil^r M +M^{\textnormal{m}})/M^{\textnormal{m}})$ is $p$-torsion free. It is equivalent to the following claim: suppose there exists $px =y +z$ for $x \in M, y \in \Fil^r M, z \in M^{\textnormal{m}}$, then $x \in \Fil^r M +M^{\textnormal{m}}$. To prove the claim, apply $\varphi$ to $px=y+z$, then $\varphi(z) = \varphi(px)- \varphi(y)$, and since $\varphi(y) \in p^rM$, thus $\varphi(z) \in M^{\textnormal{m}}\cap pM=pM^{\textnormal{m}}$. By Lemma \ref{pMm}, we have $z =pz_1 + E(u)z_2 +\frac{E(u)^p}{p} z_3, z_i \in M^{\textnormal{m}}$. Thus $p(x-z_1) = (y+\frac{E(u)^p}{p} z_3) + E(u)z_2$. To prove our claim, with some adjustment, we can just assume that we have a relation $px =y +E(u)z$ with $x \in M, y \in \Fil^r M, z \in M^{\textnormal{m}}$. Now $E(u)^{r-1}px=E(u)^{r-1}y+E(u)^{r}z\in \Fil^r M$, since $M/\Fil^r M$ is $p$-torsion free, thus $E(u)^{r-1}x \in \Fil^r M$. $\varphi(E(u)^{r-1}x) \in \varphi(\Fil^r M) \subseteq p^rM$, thus $\varphi(x) \in pM$. We now have $\varphi(E(u)z) = \varphi(px) - \varphi(y) \in p^2M$, thus $\varphi(z) \in pM^{\textnormal{m}}$. By using Lemma \ref{pMm} again, and then iterate all these steps, we will in the end have $px =y +E(u)^rz$ with $x \in M, y \in \Fil^r M, z \in M^{\textnormal{m}}$. Thus $px \in \Fil^r M$, so $x \in \Fil^r M$, and we are done. Now we have shown that $(M^{\textnormal{nil}}, \Fil^r M^{\textnormal{nil}}, \varphi_r)$ is a well-defined object in $\Mod_{\Sigma}^{\varphi}$. The nilpotency of $M^{\textnormal{nil}}$ is clear.

The second short exact sequence is by taking Cartier dual of the first sequence. Note that it is easy to verify that $M$ is \'{e}tale if and only if $M^{\vee}$ is multiplicative, by using Lemma \ref{Sigmamult}. To show that $M^{\textnormal{uni}}$ is unipotent, suppose otherwise, then there is a short exact sequence $0 \to M_1 \to M^{\textnormal{uni}} \to M_2 \to 0$ where $M_2$ is nonzero and \'{e}tale. By taking the Cartier dual of this sequence (recall that the functor of taking Cartier duals is a duality), it shows that the original nilpotent module $M^{\vee, \textnormal{nil}}$ has a nonzero multiplicative submodule, which is a contradiction.

With these two exact sequences, it is easy to verify that $M$ is nilpotent if and only if $M^{\vee}$ is unipotent, and vice versa.
\end{proof}

\begin{corollary} \label{remark}
Let $M \in \Mod_{\Sigma}^{\varphi}$ be as in Lemma \ref{SigmaFilr}, suppose $\bolda = A\bolde$ where $\bolde=\frac{\varphi_r(\bolda)}{c^r}$, then $M$ is unipotent if and only if $\Pi_{i=0}^{\infty} \varphi^i(A)=0$.
\end{corollary}
\begin{proof}
Let $A'$ be the matrix such that $AA'=E(u)^rId$. Let $M^{\vee}$ be the Cartier dual of $M$, $\bolde^{\vee}$ be the dual basis. Let $\bolda^{\vee}=A'^{T}\bolde^{\vee}$, then it is easy to show that
$\Fil^r M^{\vee} =\oplus \Sigma \bolda^{\vee} +\Fil^p \Sigma M^{\vee}$. $M$ is nilpotent if and only if $M^{\textnormal{m}} =0$, which is equivalent to $\cap_{i=0}^N (\varphi^{\ast})^i M$ converges to 0, i.e, $\varphi^N(A')\cdots \varphi(A')A' \to 0$ as $N \to \infty$. By duality, $M$ is unipotent if and only if $\Pi_{i=0}^{N} \varphi^i(A)$ converges to $0$, i.e.,
$\Pi_{i=0}^{\infty} \varphi^i(A) = 0$.
\end{proof}

\subsection{Equivalence between $\Sigma$-modules and $S$-modules}

We define a functor $\bigM_{\Sigma}: \Mod_{\Sigma}^{\varphi} \to \Mod_{S}^{\varphi}$ by setting $\M= \bigM_{\Sigma}(M) = M\otimes_{\Sigma} S$, and let $\Fil^r \M \subset \M$ be the image of $\Fil^rM \otimes_{\Sigma} S + \Fil^rS M$. By Lemma \ref{SigmaFilr}, we can write $\Fil^r M= \oplus \Sigma \bolda +\Fil^p \Sigma M$. It is easy to check that $\Fil^r \M= \oplus S \bolda +\Fil^p S \M$ and then we can define $\varphi_r$ on $\Fil^r \M$ and check that it satisfies $\varphi_r(sx)=c^{-r}\varphi_r(s)\varphi_r(E(u)^rx)$ for $s \in \Fil^r S$ and $x \in \M$.

\begin{prop}
The functor $\bigM_{\Sigma}$ is well-defined.
\end{prop}

\begin{proof}
It suffices to check that $\M/\Fil^r\M$ is $p$-torsion free. Suppose $x \in \M$ and $px \in \Fil^r \M$, suffice to show $x \in \Fil^r \M$. After adjusting $x$ by some element in $\Fil^r S M$, we can assume that $x \in M$, so $px \in M\cap \Fil^r \M$. But clearly $M \cap \Fil^r \M =\Fil^r M$, thus $px \in \Fil^r M$. Because $M/\Fil^rM$ is $p$-torsion free, $x \in \Fil^r M$, and we are done.
\end{proof}

\begin{thm} \label{SigmatoS}
The functor $\bigM_{\Sigma}: \Mod_{\Sigma}^{\varphi} \to \Mod_{S}^{\varphi}$ is an equivalence.
\end{thm}

\noindent \textbf{Conventions}: In the following proof (as well as in Subsection \ref{equiv}), we will need a lot of matrices, but the only important thing about them is where their coefficients lie. We will use notations like $Q_{n,i}$ to denote them.

\begin{proof}
\textbf{Part 1}.
We first prove that the ``$\bmod p$" functor is an equivalence, i.e, $\bigM_{\Sigma_1}: \Mod_{\Sigma_1}^{\varphi} \to \Mod_{S_1}^{\varphi}$ is an equivalence. Here $\bigM_{\Sigma_1}$ is similarly defined as $\bigM_{\Sigma}$. When $r< p-1$, we can utilize Proposition 2.2.2.1 of \cite{Bre98} to give a proof. However, Proposition 2.2.2.1 of \textit{loc. cit.} cannot be generalized to $r =p-1$ case. Indeed, the functor $T_0$ in Proposition 2.2.2.1 of \textit{loc. cit.} is not even well defined when $r=p-1$. We thank Eike Lau for pointing this out and correcting us.

We now give a unified proof for all $r \leq p-1$.

Let $M_1, M_2 \in \textnormal{Mod}_{\Sigma_1}^{\varphi}$, and $\M_1, \M_2 \in \textnormal{Mod}_{S_1}^{\varphi}$ the corresponding modules. To show full faithfulness of $\bigM_{\Sigma_1}$, it suffices to show that for any $h \in \Hom_{\Mod_{S_1}^{\varphi}}(\M_1, \M_2)$, it comes from a morphism in $\Hom_{\Mod_{\Sigma_1}^{\varphi}}(M_1, M_2)$.

Suppose $\Fil^{r} M_1 =\oplus \Sigma_1\bolda + \Fil^p \Sigma_1 M_1$, $\frac{\varphi_{r}(\bolda)}{c^{r}} = \bolde$, $\bolda =A \bolde$ with $A \in \Mat(\Sigma_1)$. Then $\Fil^{r} \M_1 =\oplus S_1\bolda + \Fil^p S_1 \M_1$. Denote $\boldb, \boldf, B$ similarly for $M_2$.

We have $h(\bolde)= T \boldf$ for $T \in \Mat(S_1)$. Since $h(\Fil^{r} \M_1) \in \Fil^{r} \M_2$, $h(\bolda) = P\boldb + (YQ_1+ Q_2)\boldf$ for some $P, Q_1 \in \Mat(\Sigma_1)$, $Q_2 \in \Mat(\Fil^{p+1}S_1)$. Because $h$ commutes with $\varphi_{r}$, we have the relation $h(\frac{\varphi_{r}(\bolda)}{c^{r}})=\frac{\varphi_{r}(h(\bolda))}{c^{r}}$, i.e., $T\boldf = \frac{\varphi_{r}(P\boldb + (YQ_1+ Q_2)\boldf)}{c^{r}}$. Since $\varphi_{r}(\Fil^{p+1}S_1)=0$ for any $r \leq p-1$, we have $T\boldf=\varphi(P) \boldf+ \frac{\varphi_{r}(Y)}{c^{r}} \varphi(Q_1) \frac{\varphi_{r}(u^{er} \boldf)}{c^{r}}$. Let $B' \in \Mat(\Sigma_1)$ be such that $BB'=u^{er} Id$, then we will have $T =\varphi(P) + \frac{\varphi_{r}(Y)}{c^{r}}\varphi(Q_1)\varphi(B')$, so $T \in \Mat(\Sigma_1)$. Since $h(\bolda)=h(A\bolde)=AT\boldf$, $PB+YQ_1+ Q_2=AT$, so $Q_2 \in \Mat(\Sigma_1)$. It is clear that $T, Q_2 \in \Mat(\Sigma_1)$ implies that $h$ comes from a morphism in $\Mod_{\Sigma_1}^{\varphi}$.

Next we show essential surjectivity of $\bigM_{\Sigma_1}$.
Let $\M \in \textnormal{Mod}_{S_1}^{\varphi}$, $\Fil^{r} \M =\oplus S_1\bolda + \Fil^p S_1 \M$, $\frac{\varphi_{r}(\bolda)}{c^{r}} = \bolde$, and $\bolda =A \bolde$ with $A \in \Mat(S_1)$. Decompose $A= P+Q$ with $P \in \Mat(\Sigma_1), Q \in \Mat(\Fil^{p+1}S_1)$. Let $\bolda'=P\bolde$, then we have $\Fil^{r} \M =\oplus S_1\bolda' + \Fil^p S_1 \M$, $\frac{\varphi_{r}(\bolda')}{c^{r}} = \bolde$. Thus $(M= \oplus \Sigma_1 \bolde, \Fil^{r} M =\oplus \Sigma_1\bolda' + \Fil^p \Sigma_1 M, \varphi_{r}) \in \textnormal{Mod}_{\Sigma_1}^{\varphi}$ and maps to $\M$.

The ``$\bmod p$" equivalence implies that the functor $\bigM_{\Sigma}$ is fully faithful by standard devissage.

\textbf{Part 2}.
Now we prove that the functor $\bigM_{\Sigma}$ is essentially surjective. For a given $\M \in \Mod_{S}^{\varphi}$, we claim that we can choose a series of $\bolda_n$ and $\bolde_n$, such that
\begin{enumerate}
\item   $\bolda_n= (\alpha_{n, 1}, \ldots, \alpha_{n, d})^T \in (\Fil^r \M)^d$, $\bolde_n = (e_{n, 1}, \ldots, e_{n, d})^T \in \M^d$;
\item   $\Fil^r \M =\oplus S \bolda_n + \Fil^p S \M$, here $\oplus S \bolda_n = \oplus_{i=1}^d S \alpha_{n, i}$;
\item   $\bolde_n = \frac{1}{c^r}\varphi_r(\bolda_n) = \frac{1}{\varphi(E(u))^{r}} \varphi(\bolda_n)$, and $\bolde_n$ is a basis of $\M$;
\item   $\bolda_n = A_n \bolde_n = (B_n +D_n)\bolde_n$, where $A_n \in \Mat_d(S)$, $B_n \in \Mat_d(\Sigma)$ and $D_n \in p^n \Mat_d(\Fil^{p+1}S)$.
\end{enumerate}

For $n=0$, this is Lemma \ref{Filr}, (and by using the fact that $S = \Sigma + \Fil^{p+1}S$).
Suppose we have done for $n$, then we take $\bolda_{n+1} = B_n \bolde_n$, and take $\bolde_{n+1}= \frac{1}{c^r}\varphi_r (\bolda_{n+1})$. Now, 

\begin{eqnarray*}
 \bolde_{n+1}  &=&     \frac{1}{c^r} \varphi_r (\bolda_n) -  \frac{1}{c^r} \varphi_r (D_n\bolde_n)\\
                 &=&    \bolde_n -  \frac{1}{c^r} \varphi_r (D_n) \frac{\varphi_r (E(u)^r \bolde_n)}{c^r}\\
                 &=&        \bolde_n -  \frac{1}{c^r} \varphi_r (D_n) \varphi(A_n') \frac{\varphi_r (A_n \bolde_n)}{c^r}, \text{here } A_n'A_n =E(u)^rId\\
                 &=&        \bolde_n -  \frac{1}{c^r} \varphi_r (D_n) \varphi(A_n')\bolde_n\\
                &=&    (Id - p^{n+1}Q_{n,1})\bolde_n, \text{ where } Q_{n,1} \in \Mat_d(S).
\end{eqnarray*}

The last line is because $p \mid \varphi_r(\Fil^{p+1}S)$, thus $\varphi_r (D_n) \in p^{n+1}\Mat_d(S)$.

From the above calculation, $\bolde_{n+1}$ is also a basis, and we have

\begin{eqnarray*}
\bolda_{n+1} &=&  B_n \bolde_n\\
               &=&  B_n  (Id-p^{n+1}Q_{n,1})^{-1} \bolde_{n+1}\\
               &=&  B_n  (Id+ p^{n+1}Q_{n,2}) \bolde_{n+1}, \text{ where } Q_{n,2} \in \Mat_d(S) \\
               &=&  B_n(  Q_{n,3} +p^{n+1}Q_{n,4} )\bolde_{n+1}, \\
               &\space& \text{ where } Q_{n,3} \in \Mat_d(\Sigma), Q_{n,4} \in \Mat_d(\Fil^{p+1}S).
\end{eqnarray*}

Now let $B_{n+1} = B_n Q_{n,3} \in \Mat_d(\Sigma)$, $D_{n+1} =B_np^{n+1}Q_{n,4} \in p^{n+1}\Mat_d(\Fil^{p+1}S)$. Thus, we have finished the construction of the algorithm.

Now, since $\bolda_{n+1} -\bolda_n =D_n\bolde_n$, and $D_n \to 0$, the sequence $\{\bolda_n\}$ converges to an $\bolda$. Let $\bolde =\frac{1}{c^r}\varphi_r(\bolda)$, then $\bolda = B\bolde$ with $B \in \Mat_d(\Sigma)$, and we still have $\Fil^r \M=\oplus S \bolda + \Fil^p S \M$.

Now take $M =\oplus \Sigma \bolde$ with $\Fil^r M =\oplus \Sigma \bolda + \Fil^p \Sigma M$. Clearly, this is the preimage of $\bigM$ under the functor $\M_{\Sigma}$.

\end{proof}

\begin{prop} \label{exactSigma}
The functor $\bigM_{\Sigma}$ and its inverse transform short exact sequences to short exact sequences.
\end{prop}
\begin{proof}
Suppose $0 \to M_1 \to M \to M_2 \to 0 $ is a short exact sequence in $\Mod_{\Sigma}^{\varphi}$, clearly the corresponding $ 0 \to \M_1 \to \M \to
\M_2 \to 0$ is short exact as finite free $S$-modules, thus we only need to check the sequence on filtrations $0 \to \Fil^r \M_1 \to \Fil^r \M \to
\Fil^r \M_2 \to 0$. The injection and the surjection are clear, thus we only need to prove the exactness on the center, i.e., $\Fil^r \M_1 = \M_1
\cap \Fil^r \M$. It is clear that $\Fil^r \M_1 \subseteq \M_1 \cap \Fil^r \M$. Suppose $x \in \M_1 \cap \Fil^r \M$, after adjusting by some element in $\Fil^p S \M_1$, we can suppose $x \in M_1 \cap \Fil^r \M$ (since $\Fil^p S\M_1 \subseteq  \Fil^r \M$). Now,

\begin{eqnarray*}
  x &\in&  (M_1 \cap  M) \cap  \Fil^r \M   =  M_1 \cap (M \cap  \Fil^r \M )  \\
   &=& M_1 \cap \Fil^r M = \Fil^r M_1,
\end{eqnarray*}
thus we are done.

The proof for the inverse functor is similar.
\end{proof}

\begin{corollary} \label{Sexact}
An $S$-module $\M \in \Mod_{S}^{\varphi}$ is \'{e}tale, multiplicative, nilpotent, or unipotent if and only if its corresponding $\Sigma$-module $M$ is so. The Cartier dual of these two modules are compatible with each other via the functor $\bigM_{\Sigma}$. And we have short exact sequences in the category $\Mod_{S}^{\varphi}$ like in Theorem \ref{SigmaM} (by tensoring with $S$).
\end{corollary}

\begin{proof}
This is an easy consequence of Theorem \ref{SigmatoS} and Proposition \ref{exactSigma}.
\end{proof}

In the above corollary, we defined $\M^{\rm{m}} := M^{\rm m} \otimes_{\Sigma} S$. We did not define $\M^{\rm m}$ as $\cap_{n=1}^{\infty} (\varphi^{\ast})^{n}\M$ where $(\varphi^{\ast})^n \M$ is the $S$-span of $\varphi^n(\M)$. The reason is because it seems impossible to directly prove that $\cap_{n=1}^{\infty} (\varphi^{\ast})^{n}\M$ is finite free over $S$ as in Theorem \ref{Mm}. But in fact, we can now prove these two definitions are equivalent.

\begin{prop} \label{bigMm}
Let $\M \in \Mod_{S}^{\varphi}$, and $M \in \Mod_{\Sigma}^{\varphi}$ such that $\M_{\Sigma}(M)=\M$. Then $\cap_{n=1}^{\infty} (\varphi^{\ast})^{n}\M = \M^{\rm m} =M^m \otimes S$.
\end{prop}
\begin{proof}
It is clear that $\cap_{n=1}^{\infty} (\varphi^{\ast})^{n}\M \supseteq M^m \otimes S$. It suffices to prove the other direction.

Let $$0 \to M^{\rm m} \to M \to M^{\rm nil} \to 0$$ be the short exact sequence as in Theorem \ref{SigmaM}. Let $(e_1, \ldots, e_{d_1})$ be a basis of $M^{\rm m}$, $(f_1, \ldots, f_{d_2})$ a basis of $M^{\rm nil}$. Let $\hat{f}_j \in M$ be a lift of $f_j$ for $1 \leq j \leq d_2$, then $(e_1, \ldots, e_{d_1}, \hat{f}_1, \ldots, \hat{f}_{d_2})$ is a basis of $M$.
Since $M^{\rm nil}$ is nilpotent, by Corollary \ref{remark}, $\varphi^n(f_j) \to 0, \forall j$ as $n \to \infty$. Thus for any $k>0$, we can choose $n_k$ big enough such that $\varphi^{n_k}(f_j) =p^k y_{j,k}$ for some $y_{j,k} \in M^{\rm nil}$.
Since $\oplus_{l=1}^{d_2} \Sigma \hat{f}_l$ maps surjectively to $M^{\rm nil}$, we can choose some $\sum_{l=1}^{d_2} b_{j,k,l}\hat{f}_l$ that maps to $y_{j,k}$. Note that $\varphi^{n_k}(\hat{f}_j)$ maps to $\varphi^{n_k}(f_j)$ , so $\varphi^{n_k}(\hat{f}_j) - p^k\sum_{l=1}^{d_2} b_{j,k,l}\hat{f}_l$ maps to $0$, which means that $\varphi^{n_k}(\hat{f}_j) - p^k\sum_{l=1}^{d_2} b_{j,k,l}\hat{f}_l \in M^{\rm m}$. Thus we have $\varphi^{n_k}(\hat{f}_j) \in p^k(\oplus_{l=1}^{d_2} \Sigma \hat{f}_l) +M^{\rm m} \subseteq p^kM+M^{\rm m}$.

Now we have
\begin{eqnarray*}
(\varphi^{\ast})^{n_k}\M  &=&\{ \varphi^{n_k}(e_1), \ldots, \varphi^{n_k}(e_{d_1})  \}_S + \{ \varphi^{n_k}(\hat{f}_1), \ldots, \varphi^{n_k}(\hat{f}_{d_2})  \}_S\\
&=& \M^{\rm m} +\{ \varphi^{n_k}(\hat{f}_1), \ldots, \varphi^{n_k}(\hat{f}_{d_2})  \}_S\\
&\subseteq&  \M^{\rm m}+(p^k\M+\M^{\rm m})\\
&=&   \M^{\rm m}+p^k\M,
\end{eqnarray*}
where $\{?\}_S$ denotes the linear $S$-span of elements inside.

Clearly, to prove the proposition, it suffices to show that $\cap_{k=0}^{\infty} (\M^{\rm m}+p^k\M) =\M^{\rm m}$. Let $x \in \cap_{k=0}^{\infty} (\M^{\rm m}+p^k\M)$. Then for any $k$, $x=x_k+p^ky_k$ for some $x_k \in \M^{\rm m}, y_k \in \M$. The sequence $\{x-x_k\}_{k=0}^{\infty}$ clearly converges to $0$, which means that $\{x_k\}_{k=0}^{\infty}$ converges to $x$. Thus $x \in \M^{\rm m}$, and we finish the proof of the proposition.

\end{proof}

\subsection{Equivalence between unipotent $\huaS$-modules and $S$-modules} \label{equiv}

In this part, we prove the equivalence of $\huaS$-modules and $\Sigma$-modules (equivalently, $S$-modules) in the $r=p-1$ unipotent case.

We have an injective map of $W(k)$-algebras $\huaS \hookrightarrow S$ by $u \mapsto u$. Let $\varphi: \huaS \to S$ be the map obtained by composing the injection and the Frobenius on $\huaS$.
We define the functor $\M_{\huaS}$ from $\Mod_{\huaS}^{\varphi}$ to $\Mod_{S}^{\varphi}$ as follows. Let $\huaM \in \Mod_{\huaS}^{\varphi}$, set $\M = \M_{\huaS}(\huaM)=S\otimes_{\varphi,\huaS} \huaM$. We have an $S$-linear map $1 \otimes \varphi: S\otimes_{\varphi,\huaS} \huaM \to S\otimes_{\huaS} \huaM$.  Set
$$\Fil^{r}\M=\{x\in\M,(1\otimes\varphi)(x)\in \Fil^{r}S\otimes_{\huaS}\huaM \subseteq S\otimes_{\huaS}\huaM\},$$
and define $\varphi_r: \Fil^r \M \to \M$ as the composite:
$$\Fil^r \M \stackrel{1\otimes \varphi}{\longrightarrow} \Fil^r S \otimes_{\huaS} \huaM \stackrel{\varphi_r\otimes 1}{\longrightarrow} S\otimes_{\varphi, \huaS} \huaM =\M.$$
It can be easily checked that the above functor is well-defined.
On the representation level, by the natural embedding $\iota: \huaS^{\text{ur}} \hookrightarrow A_{\text{cris}}$, we have the natural map
$$ T_{\huaS}(\huaM)= \Hom_{\huaS, \varphi} (\huaM, \huaS^{\text{ur}}) \to \Hom_{'\Mod_{S}^{\varphi}} (\M_{\huaS}(\huaM), A_{\text{cris}})=T_{\textnormal{cris}}(\M_{\huaS}(\huaM)).$$

\begin{rem}
We can also define the functor $M_{\huaS}$ from $\Mod_{\huaS}^{\varphi}$ to $\Mod_{\Sigma}^{\varphi}$ analogously.
\end{rem}

\begin{thm}[\cite{Liu08}, \cite{CL09}]
When $r<p-1$, the functor $\M_{\huaS}$ induces an equivalence between $\Mod_{\huaS}^{\varphi}$ and $\Mod_{S}^{\varphi}$, and the equivalence is compatible with respect to Galois representations, namely $T_{\huaS}(\huaM) = T_{\textnormal{cris}}(\M_{\huaS}(\huaM))$ as finite free $\Zp$-representation of $G_{\infty}$.
\end{thm}

We will generalize the above theorem to the $r =p-1$ unipotent case. Before stating the theorem, we introduce the following definition.

\begin{defn} \label{modpunipdef}
Let $r=p-1$.
\begin{enumerate}
\item  For $\huaM \in \Mod_{\huaS_1}^{\varphi}$, let $\bolde$ be a basis, $\varphi(\bolde)= A\bolde$, and $A' \in \Mat(\huaS_1)$ such that $AA' = u^{er}Id$. $\huaM$ is called unipotent (with respect to the basis $\bolde$) if $\prod_{n=1}^{\infty} \varphi^n(A')=0$.

\item  For $\M \in \Mod_{S_1}^{\varphi}$, suppose we have $\Fil^r \M =\oplus\bolda + \Fil^p S_1 \M$, $\frac{\varphi_{r}(\bolda)}{c^{r}} = \bolde$, $\bolda=A \bolde$ with $A \in \Mat(S_1)$. $\M$ is called unipotent (with respect to $\bolda$ and $\bolde$) if $\prod_{n=1}^{\infty} \varphi^n(A)=0$.

\item For $M \in \Mod_{\Sigma_1}^{\varphi}$, define unipotency similarly as for $\Mod_{S_1}^{\varphi}$.
\end{enumerate}
\end{defn}

\begin{lemma}
The above definition of unipotency is independent of choice of $\bolde$ (resp. $\bolda$ and $\bolde$).
\end{lemma}

\begin{proof}
The case for $\huaM \in \Mod_{\huaS_1}^{\varphi}$ is trivial to check. Now we check for $\Mod_{S_1}^{\varphi}$.

For $\M \in \Mod_{S_1}^{\varphi}$, suppose we also have $\Fil^r \M =\oplus \boldb + \Fil^p S_1 \M$, $\frac{\varphi_{r}(\boldb)}{c^{r}} = \boldf$, $\boldb =B \boldf$ with $B \in \Mat(S_1)$. It suffices to check that $\prod_{n=1}^{\infty} \varphi^n(B)=0$.

Suppose $\bolda =P_1 \boldb +Q_1 \boldf$, $\boldb =P_2 \bolda +Q_2 \bolde$ with $P_1, P_2 \in \Mat(\huaS_1), Q_1, Q_2 \in \Mat(\Fil^p S_1)$. Then
$$\bolde=\frac{\varphi_r(\bolda)}{c^r} =\frac{\varphi_r(P_1 \boldb +Q_1 \boldf)}{c^r}= \varphi(P_1)\boldf + \frac{\varphi_r(Q_1)}{c^r}\varphi(B')\boldf,$$
where $B' \in \Mat(S_1)$ such that $BB' =u^{er}Id$. Let $T=\varphi(P_1)+ \frac{\varphi_r(Q_1)}{c^r}\varphi(B')$, then $\bolde=T\boldf$. Now $A\bolde=\bolda=P_1(P_2A\bolde+Q_2\bolde)+Q_1T^{-1}\bolde$, so $A=P_1P_2A+P_1Q_2+Q_1T^{-1}$. Apply $\varphi$ on both sides, we have $\varphi(A)=\varphi(P_1P_2A)$.

We have $\boldb=B\boldf=P_2A\bolde+Q_2\bolde=(P_2A+Q_2)T\boldf$, so $B=(P_2A+Q_2)T$.

Then
\begin{eqnarray*}
\varphi(B)\varphi^2(B)  &=& \varphi(P_2)\varphi(A)\varphi(T)\varphi^2(B) \\
  &=& \varphi(P_2)\varphi(A)(\varphi^2(P_1)+ \varphi(\frac{\varphi_r(Q_1)}{c^r})\varphi^2(B'))\varphi^2(B) \\
&=& \varphi(P_2)\varphi(A)\varphi^2(P_1)\varphi^2(B), \text{using } \varphi(B'B)=0 \\
&=&  \varphi(P_2)\varphi(A)\varphi^2(P_1)(\varphi^2(P_2)\varphi^2(A))\varphi^2(T) \\
&=& \varphi(P_2)\varphi(A)\varphi^2(A)\varphi^2(T), \text{using } \varphi(P_1P_2A)=\varphi(A)
\end{eqnarray*}

By iterating the above process, we will see that $\prod_{n=1}^N \varphi^n(B) =\varphi(P_2) \prod_{n=1}^N \varphi^n(A)\varphi^N(T)$, so it converges to 0.
\end{proof}

\begin{rem}
Note that we could actually define the unipotency on $\huaM \in \Mod_{\huaS_1}^{\varphi}$ by requiring it to have no nonzero \'{e}tale quotient, and it is easy to check it is equivalent to the definition above. It would be desirable if we could define the unipotency on $\Mod_{S_1}^{\varphi}$ (or $\Mod_{\Sigma_1}^{\varphi}$) similarly as in Definition \ref{defunip}. However we are unable to achieve this because we do not know if we can prove $\Sigma_1$-version of Theorem \ref{SigmaM}. Fortunately, Definition \ref{modpunipdef} is enough for our application.
\end{rem}

\begin{thm} \label{mtoM}
When $r=p-1$, the functor $\bigM_{\huaS}$ induces an equivalence between the category of unipotent $\huaS$-modules and the category of unipotent $S$-modules, and $T_{\huaS}(\huaM) = T_{\textnormal{cris}}(\M_{\huaS}(\huaM))$ as finite free $\Zp$-representation of $G_{\infty}$.
\end{thm}

Before we prove the theorem, we record a useful lemma.

\begin{lemma}\label{huaS}
For a matrix $A \in \Mat_d(\Sigma)$, let $A=B+C$ where $B \in \Mat_d(\huaS), C \in \Mat_d(\Fil^p \Sigma)$. If there exists some $A' \in \Mat_d(\Sigma)$ such that $AA' =  Id$ (resp. $AA' =  E(u)^rId$), then there exists $B' \in \Mat_d(\huaS)$ such that $BB'=Id$ (resp. $BB'=E(u)^rId$).
\end{lemma}

\begin{proof}
Suppose we have $AA' =Id$. Write $A' =B_1 + C_1$ with $B_1 \in \Mat_d(\huaS), C_1 \in \Mat_d(\Fil^p\Sigma)$, then $(B+C)(B_1+C_1)=Id$. So $BB_1-Id= -BC_1 - CB_1- CC_1$, with the left hand side $\in \Mat_d(\huaS)$ while the right hand side $\in \Mat_d(\Fil^p \Sigma)$. Since $\huaS \cap \Fil^p\Sigma =E(u)^p\huaS$, $BB_1-Id = E(u)^pD$, with $D \in \Mat_d(\huaS)$. Thus we can take $B' = B_1(Id+E(u)^pD)^{-1}$.

The case when $AA' =E(u)^rId$ is similar.
\end{proof}

\begin{rem}
If we change $(\Sigma, \huaS, \Fil^p \Sigma)$ in the setting of the above lemma to $(S, \huaS, \Fil^p S)$, or $(S, \Sigma, \Fil^{p+1}S)$, or $(S[1/p], \Sigma[1/p], \Fil^{p+1} S[1/p])$, or $(S_1, \huaS_1, \Fil^p S_1)$, the corresponding variant lemma is still valid.
\end{rem}

\begin{proof}[Proof of Theorem \ref{mtoM}]

\textbf{Part 0}. We show that the functor is well-defined, namely, if $\huaM$ is unipotent, then $M_{\huaS}(\huaM)$ (thus $\M_{\huaS}(\huaM)$, via Corollary \ref{Sexact}) is also unipotent.

Suppose that the matrix of $\varphi$ on $\huaM$ for a fixed basis $\bolde$ is $A$. Then as remarked in the proof of Proposition \ref{huaMexact}, $\huaM$ is unipotent if and only if $\Pi_{n=1}^{\infty} \varphi^n(A') =0$, where $A'$ is the matrix such that $AA'=E(u)^rId$. Let $M=M_{\huaS}(\huaM)$. Then we have $\Fil^r M = \oplus S \bolda + \Fil^p S M$ with the basis for $M$ being $1\otimes \bolde$, and $\bolda = A'(1\otimes \bolde)$. Thus, by Corollary \ref{remark}, $M$ is unipotent.

\textbf{Part 1}. We show that the functor is fully faithful.

To do this, it suffices to show that the ``mod $p$ functor" $\M_{\huaS_1} : \Mod_{\huaS_1}^{\varphi} \to \Mod_{S_1}^{\varphi}$ induces an equivalence between the unipotent subcategories. By unipotent subcategory, we mean the subcategory consisting of unipotent objects. Here $\M_{\huaS_1}$ is defined analogously as $\M_{\huaS}$.

Theorem 4.1.1 of \cite{Bre99b} proved this ``mod $p$" equivalence for $r< p-1$ and $e=1$ (without unipotency condition), and the proof can be directly generalized to arbitrary $e$ case. But the proof cannot generalize to $r=p-1$ case because it relies on Proposition 2.2.2.1 of \cite{Bre98}, which, as we have mentioned in the proof of Theorem \ref{SigmatoS}, cannot be generalized to $r=p-1$ case .

So now let us suppose $r=p-1$. Let $\huaM_1, \huaM_2 \in \textnormal{Mod}_{\huaS_1}^{\varphi}$ be unipotent modules, and $\M_1, \M_2 \in \textnormal{Mod}_{S_1}^{\varphi}$ the corresponding modules (which are clearly unipotent). To show full faithfulness of $\bigM_{\huaS_1}$, it suffices to show that for any $h \in \Hom_{\Mod_{S_1}^{\varphi}}(\M_1, \M_2)$, it comes from a morphism in $\Hom_{\Mod_{\huaS_1}^{\varphi}}(\huaM_1, \huaM_2)$.

Let $\bolde$ be a basis of $\huaM_1$ such that $\varphi(\bolde)=A'\bolde$, and $A \in \Mat(\huaS_1)$ such that $AA'=u^{er}Id$. Then we have $\Fil^{r} \M_1 =\oplus S_1\bolda + \Fil^p S_1 \M_1$, $\frac{\varphi_{r}(\bolda)}{c^{r}} = \bolde$, where $\bolda =A \bolde$. Denote $\boldb, \boldf, B$ similarly for $\M_2$. We have $h(\bolde)= T \boldf$ for $T \in \Mat(S_1)$. Since $h(\Fil^{r} \M_1) \in \Fil^{r} \M_2$, $h(\bolda) = P\boldb + (YQ_1 +Q_2)\boldf$ for $P, Q_1 \in \Mat(\huaS_1), Q_2 \in \Mat(\Fil^{p+1}S_1)$. Because $h$ commutes with $\varphi_{r}$, we have the relation $T\boldf=\varphi(P) \boldf+ \frac{\varphi_{r}(Y)}{c^{r}} \varphi(Q_1) \frac{\varphi_{r}(u^{er} \boldf)}{c^{r}}$. Let $B' \in \Mat(\huaS_1)$ be such that $BB'=u^{er} Id$, then we will have $T =\varphi(P) + \frac{\varphi_{r}(Y)}{c^{r}}\varphi(Q_1)\varphi(B')=\varphi(P) + c\varphi(Q_1)\varphi(B')$.

Since $h(\bolda)=h(A\bolde)=AT\boldf$, so $AT=A\varphi(P) + (Y+t)A\varphi(Q_1)\varphi(B')= PB +YQ_1 +Q_2$, here $c=Y+t$ for $t \in \huaS_1$. Since $\huaS_1 \cap \Fil^p S_1 =0$, we must have $A\varphi(P)+ tA\varphi(Q_1)\varphi(B')= PB$,  $YA\varphi(Q_1)\varphi(B')= YQ_1 +Q_2$. Apply $\varphi_{r}$ on both sides of the second equation, we get $\varphi(A)\varphi^2(Q_1)\varphi^2(B')=\varphi(Q_1)$. By iteration of the preceding equality, we will have $\varphi(Q_1)= \prod_{k=1}^n \varphi^k(A)\varphi^{n+1}(Q_1) \prod_{k=n+1}^2 \varphi^k(B')$ for any $n$. Thus it is equal to $0$ for $n$ big enough by unipotency assumption ($\prod_{k=1}^n \varphi^k(A) \to 0$). So we have $YQ_1 +Q_2=0$, and $T=\varphi(P) \in \Mat(\huaS_1)$. Thus we have $h(\bolde)= T \boldf$, $h(\bolda) = P\boldb$ for $T, P \in \Mat(\huaS_1)$, it is clear that $h$ comes from a morphism in $\Hom_{\huaS_1}(\huaM_1, \huaM_2)$.

Now we prove essential surjectivity of $\bigM_{\huaS_1}$. It is equivalent to show that $M_{\huaS_1}:\textnormal{Mod}_{\huaS_1}^{\varphi} \to \textnormal{Mod}_{\Sigma_1}^{\varphi}$ is essentially surjective by Theorem \ref{SigmatoS}.

Given a unipotent $M \in \Mod_{\Sigma_1}^{\varphi}$, we claim that we can choose a series of $\bolda_n$ and $\bolde_n$, such that
\begin{enumerate}
\item   $\bolda_n \in (\Fil^r M)^d$, $\bolde_n \in M^d$;

\item   $\Fil^r M =\oplus \Sigma_1 \bolda_n + \Fil^p \Sigma_1 M$;

\item   $\bolde_n = \frac{1}{c^r}\varphi_r(\bolda_n)$, and $\bolde_n$ is a basis of $M$;

\item   $\bolda_n = A_n \bolde_n = (B_n + C_n Y+D_n)\bolde_n$, where $A_n \in \Mat_d(\Sigma_1)$, $B_n, C_n \in \Mat_d(\huaS_1)$, $D_0 \in \Mat_d(\Fil^{p+1}\Sigma_1)$ and $D_n=0$ for $n \geq 1$.
\end{enumerate}

We claim that with this algorithm, $C_n$ will become 0 for $n$ big enough.

Starting from $n=0$, just apply Lemma \ref{SigmaFilr}. Note that by the fact that any element $a \in \Sigma_1$ can be expressed as $a=b+cY+d$ with $b, c \in \huaS_1, d \in \Fil^{p+1}\Sigma_1$, we can decompose $A_0=B_0 + C_0 Y+D_0$.

Suppose we have done for $n$. Then we take $\bolda_{n+1} = B_n \bolde_n$. Since $C_n Y + D_n \in \Mat_d(\Fil^p \Sigma_1)$, $\bolda_{n+1} \in (\Fil^r M)^d$, and $\Fil^r M =\oplus \Sigma_1 \bolda_{n+1} + \Fil^p \Sigma_1 M$.
And we set $\bolde_{n+1}= \frac{1}{c^r}\varphi_r (\bolda_{n+1})$. By the same argument as in Lemma 2.2.3, $\bolde_{n+1}$ is a basis for $M$.

Now $\bolde_{n+1} =\frac{\varphi_{r}(\bolda_n - (C_nY +D_n)\bolde_n)}{c^{r}} =  \bolde_n  - c\varphi(C_n)\varphi(B_n')\bolde_{n+1} $ where $B_nB_n'=u^{er} Id$ ($B_n'$ exists by Lemma \ref{huaS}). Thus $\bolde_n =(1+c\varphi(C_n)\varphi(B_n'))\bolde_{n+1}$. So $\bolda_{n+1} =B_n \bolde_n= B_n(1+c\varphi(C_n)\varphi(B_n'))\bolde_{n+1}$. That is $A_{n+1}= B_n + tB_n\varphi(C_n)\varphi(B_n') + YB_n\varphi(C_n)\varphi(B_n')$.

Now we set $B_{n+1}=B_n + tB_n\varphi(C_n)\varphi(B_n')$, $C_{n+1}=B_n\varphi(C_n)\varphi(B_n')$ and $D_{n+1}=0$. And we are done for the construction of the algorithm.

Now by iterating the relation $C_{n+1}=B_n\varphi(C_n)\varphi(B_n')$, we will have $C_{n+1} =
B_n\varphi(C_n)\varphi(B_n') =B_n \varphi(B_{n-1}) \varphi^2(C_{n-1}) \cdot \ast =\ldots$. We can see that the ``front part" of $C_{n+1}$ is $\Pi_{i=0}^{n} \varphi^{i}(B_{n-i})$. Note that $B_{i+1}\varphi(B_i) = (B_i +C_{i+1} t) \varphi(B_i)$, and $C_{i+1}\varphi(B_i)=B_i\varphi(C_i)\varphi(B_i')\varphi(B_i)=B_i\varphi(C_i)\varphi(u^{er})=0$, so $B_{i+1}\varphi(B_i) = B_i \varphi(B_i)$. So
 $\Pi_{i=0}^{n} \varphi^{i}(B_{n-i})=\Pi_{i=0}^{n} \varphi^{i}(B_0)$. Now $\Pi_{i=1}^{n} \varphi^{i}(B_0)=\Pi_{i=1}^{n} \varphi^{i}(A_0)$ because $\varphi(C_0Y+D_0)=0$, and it will be equal to 0 for $n$ big enough by unipotency condition. Thus, $C_n=0$ for $n$ big enough.

We have now found some suitable basis such that $\Fil^r M =\oplus \Sigma_1 \bolda + \Fil^p \Sigma_1 M$, $\bolda = A \bolde$ with $A \in \Mat_d(\huaS_1)$, and $\bolde = \frac{1}{c^r}\varphi_r(\bolda)$ is a basis of $M$. Let $\huaM = \oplus \huaS_1 \bolde$, $\varphi(\bolde)=A' \bolde$ where $A' \in \Mat_d(\huaS_1)$ such that $AA'=u^{er}Id$. Then it is a unipotent module in $\textnormal{Mod}_{\huaS_1}^{\varphi}$ and maps to $M$.

Now, by standard devissage, $\M_{\huaS}$ is fully faithful.

\textbf{Part 2}. We show that the functor $M_{\huaS}$ (thus $\M_{\huaS}$) is essential surjective. We will improve the proof of Lemma 2.2.2 of \cite{CL09}. (Note that in the following, $r =p-1$.)

Given a unipotent module $M \in \Mod_{\Sigma}^{\varphi}$, we claim that we can choose a series of $\bolda_n$ and $\bolde_n$, such that
\begin{enumerate}

\item   $\bolda_n \in (\Fil^r M)^d$, $\bolde_n \in M^d$;

\item   $\Fil^r M =\oplus \Sigma \bolda_n + \Fil^p \Sigma M$;

\item   $\bolde_n = \frac{1}{c^r}\varphi_r(\bolda_n) = \frac{1}{\varphi(E(u))^{p-1}} \varphi(\bolda_n)$, and $\bolde_n$ is a basis of $M$;

\item   $\bolda_n = A_n \bolde_n = (B_n + C_n Y+D_n)\bolde_n$, where $A_n \in \Mat_d(\Sigma)$, $B_n, C_n \in \Mat_d(\huaS)$ and $D_n \in \Mat_d(\Fil^{p+1}\Sigma)$, here $Y = \frac{E(u)^p}{p}$.

\end{enumerate}

We claim that in our algorithm, $C_n$ will be divisible by $p$ in $\Mat_d(\Sigma)$ for $n$ big enough.

Starting from $n=0$, this is Lemma \ref{SigmaFilr}, and by using the fact that any element $a \in \Sigma$ can be expressed as $a=b+cY+d$ with $b, c \in \huaS, d \in \Fil^{p+1}\Sigma$.

Suppose we have done for $n$. Then we take $\bolda_{n+1} = B_n \bolde_n$. Since $C_n Y + D_n \in \Mat_d(\Fil^p \Sigma)$, $\bolda_{n+1} \in (\Fil^r M)^d$, and $\Fil^r M =\oplus \Sigma \bolda_{n+1} + \Fil^p \Sigma M$.
And we set $\bolde_{n+1}= \frac{1}{c^r}\varphi_r (\bolda_{n+1})$. By the same argument as in Lemma \ref{Filr}, $\bolde_{n+1}$ is a basis for $M$. And we have,

\begin{eqnarray*}
 \bolde_{n+1}  &=& \frac{1}{\varphi(E(u)^{p-1})}  \varphi(\bolda_{n+1}) \\
    &=&  \frac{1}{\varphi(E(u)^{p-1})} \varphi(B_n) \varphi(\bolde_n)\\
    &=&  \frac{1}{\varphi(E(u)^{p-1})} \varphi(B_n) \frac{\varphi(A_n')\varphi(A_n e_n)}{\varphi(E(u)^{p-1})}, \text{here } A_n'A_n =E(u)^rId \\
    &=&  \frac{\varphi(B_n A_n')}{\varphi(E(u)^{p-1})}  \frac{\varphi(\bolda_n)}{\varphi(E(u)^{p-1})}\\
    &=&   \frac{\varphi(B_n A_n')}{\varphi(E(u)^{p-1})}  \bolde_n
 \end{eqnarray*}

Thus, $\bolda_{n+1}= B_n \bolde_n = B_n \frac{1}{\varphi(E(u)^{p-1})} \varphi(A_n) \varphi(B_n') \bolde_{n+1}$. Here $B_n B_n'= E(u)^{p-1}Id$, note that $B_n' \in \Mat_d(\huaS)$ exists via Lemma \ref{huaS}.

So,
\begin{eqnarray*}
A_{n+1} &=& \frac{1}{\varphi(E(u)^{p-1})} B_n \varphi(A_n) \varphi(B_n')\\
        &=&  B_n \frac{\varphi(B_n +C_nY+D_n) \varphi(B_n')}{ \varphi(E(u)^{p-1})}\\
        &=&  B_n \varphi(\frac{B_n B_n' + C_n B_n'Y+ D_n B_n'}{E(u)^{p-1}})\\
        &=&  B_n (\varphi(Id + C_n B_n' \frac{E(u)}{p}) + p Q_{n,1}), \\
        &\space& \text{ here because } p \mid \varphi(\frac{\Fil^{p+1}\Sigma}{E(u)^{p-1}}), \varphi(\frac{D_n B_n'}{E(u)^{p-1}})=p Q_{n, 1} \text{ with } Q_{n,1} \in \Mat_d(\Sigma)\\
        &=& B_n + B_n\varphi(C_n)\varphi(B_n')(Y+t)+p Q_{n,2}, Q_{n,2} \in \Mat_d(\Sigma), \\
        &\space &\text{ here } \varphi(E(u)/p)=c=Y+t, \text{ with } t \in \huaS\\
        &=& B_n + Q_{n,3} + B_n\varphi(C_n)\varphi(B_n')t +B_n\varphi(C_n)\varphi(B_n')Y+Q_{n,4},\\
        &\space & \text{ where } p Q_{n,2}=Q_{n,3}+Q_{n,4}, \text { with } Q_{n,3} \in \Mat_d(\huaS), Q_{n,4} \in \Mat_d(\Fil^{p+1}\Sigma)
\end{eqnarray*}

The last step uses the fact that any element $a \in p\Sigma$ can be expressed as $a=b+c$ for some $b \in \huaS$ and $c \in \Fil^{p+1}\Sigma$. Also we have $Q_{n,3} \in p\Mat_d(\Sigma)$.

Now we can choose $C_{n+1} =B_n\varphi(C_n)\varphi(B_n') $, $B_{n+1} = B_n +C_{n+1}t +Q_{n,3}$, and $D_{n+1}=Q_{n,4}$. And we are done for the construction of the algorithm.

Now by iterating the relation $C_{n+1} =B_n\varphi(C_n)\varphi(B_n') $, we will have $C_{n+1} =
B_n\varphi(C_n)\varphi(B_n') =B_n \varphi(B_{n-1}) \varphi^2(C_{n-1}) \cdot \ast =\ldots$. We can see that the ``front part" of $C_{n+1}$ is $\Pi_{i=0}^{n} \varphi^{i}(B_{n-i})$.
We claim that $C_{n+1}$ is divisible by $p$ in $\Mat_d(\Sigma$) for $n$ big enough. Note that $B_i
\varphi(B_{i-1}) = (B_{i-1} +C_i t+Q_{i-1,3}) \varphi(B_{i-1})$, $C_i \varphi(B_{i-1}) =B_{i-1} \varphi(C_{i-1}) \varphi(B_{i-1}')
\varphi(B_{i-1}) = B_{i-1} \varphi(C_{i-1}) \varphi(E(u)^{p-1})$ is divisible by $p$, and $Q_{i-1,3}$ is also divisible by $p$. Thus the
divisibility of $B_i \varphi(B_{i-1})$ by $p$ is the same as that of $B_{i-1} \varphi(B_{i-1})$! By iterating this process, we see that the
divisibility of $ \Pi_{i=0}^{n} \varphi^{i}(B_{n-i})$ is the same as that of $ \Pi_{i=0}^{n} \varphi^{i}(B_0)$. Now that the divisibility of  $
\Pi_{i=1}^{n} \varphi^{i}(B_0)$ is the same as that of  $ \Pi_{i=1}^{n} \varphi^{i}(A_0)$, since $p\mid \varphi(C_0 Y+D_0)$. But $ \Pi_{i=0}^{n}
\varphi^{i}(A_0)$ converges to $0$ because $M$ is unipotent (by Corollary \ref{remark})! And we finish our proof that $C_{n+1}$ is divisible by $p$ for $n$ big enough.

Now, we will do a similar iteration, where we choose a series of $\bolda_n$ and $\bolde_n$, such that,
\begin{enumerate}
\item $\bolda_n \in (\Fil^r M)^d$, $\bolde_n \in M^d$;
\item   $\Fil^r M =\oplus \Sigma \bolda_n + \Fil^p \Sigma M$;
\item   $\bolde_n = \frac{1}{c^r}\varphi_r(\bolda_n)$, and $\bolde_n$ is a basis of $M$;
\item   $\bolda_n = A_n \bolde_n = (B_n + D_n)\bolde_n$, where $A_n \in \Mat_d(\Sigma)$, $B_n \in \Mat_d(\huaS)$ and $D_n \in p^n\Mat_d(\Fil^{p+1}\Sigma)$.
\end{enumerate}

For $n=0$, it's just what we have proven above, because $C_N$ is divisible by $p$ for $N$ large enough, and $C_N Y$ can break into $\huaS$-part and $\Fil^{p+1}\Sigma$-part.
Suppose we have done for $n$,
then we take $\bolda_{n+1} = B_n \bolde_n$, and take $\bolde_{n+1}= \frac{1}{c^r}\varphi_r (\bolda_{n+1})$,
now
\begin{eqnarray*}
 \bolde_{n+1}  &=&     \frac{1}{c^r} \varphi_r (\bolda_n) -  \frac{1}{c^r} \varphi_r (D_n\bolde_n)\\
                 &=&    \bolde_n -  \frac{1}{c^r} \varphi_r (D_n) \frac{\varphi_r (E(u)^r \bolde_n)}{c^r}\\
                 &=&        \bolde_n -  \frac{1}{c^r} \varphi_r (D_n) \varphi(A_n') \frac{\varphi_r (A_n \bolde_n)}{c^r}\\
                &=&    (Id - p^{n+1}Q_{n,5})\bolde_n, \text{ where } Q_{n,5} \in \Mat_d(\Sigma), \text{because } \varphi_r (D_n) \in p^{n+1}\Mat_d(\Sigma)
\end{eqnarray*}

Thus,
\begin{eqnarray*}
\bolda_{n+1} &=&  B_n \bolde_n\\
               &=&  B_n  (Id-p^{n+1}Q_{n,5})^{-1} \bolde_{n+1}\\
               &=&  B_n  (Id+ p^{n+1}Q_{n,6}) \bolde_{n+1}, Q_{n,6} \in \Mat_d(\Sigma)\\
               &=&  B_n(Q_{n,7} +p^{n+1}Q_{n,8}  )\bolde_{n+1}, \\
               &\space& \text{ where } Q_{n,7} \in \Mat_d(\huaS),  Q_{n,8} \in \Mat_d(\Fil^{p+1}\Sigma)
\end{eqnarray*}

Now we set $B_{n+1} = B_n  Q_{n,7}, D_{n+1} =B_np^{n+1} Q_{n,8} \in p^{n+1}\Mat_d(\Fil^{p+1}\Sigma)$. Thus, we have finished the construction of the algorithm.

Now, since $\bolda_{n+1} -\bolda_n =D_n\bolde_n$, and $D_n \to 0$, the sequence $\{\bolda_n\}$ converges to an $\bolda$. Let $\bolde =\frac{1}{c^r}\varphi_r(\bolda)$, then $\bolda = B\bolde$ with $B \in \Mat_d(\huaS)$, and we still have $\Fil^r M=\oplus \Sigma \bolda + \Fil^p \Sigma M$.
Now by Lemma \ref{huaS}, there exists $B' \in \Mat_d(\huaS)$ with $BB' =(E(u))^rId$. Define $\huaM =\oplus_{i=1}^d \huaS f_i$ with $\varphi (f_1, \ldots, f_d)^T =B' (f_1, \ldots, f_d)^T$, then it is easy to check that $\huaM$ is the preimage of $M$ under the functor $M_{\huaS}$.

\textbf{Part 3}. Compatibility with Galois representations.

This is Proposition 1.2.7 of \cite{Kis09b}.

\end{proof}

\section{Strongly divisible lattices and unipotency}

\subsection{Strongly divisible lattices}
In this section, we review the notion of strongly divisible lattices. We will define unipotency on them, and prove that in a unipotent semi-stable Galois representation, all quasi-strongly divisible lattices are unipotent.

Let $S_{K_0}:= S\otimes_{\Zp}\Qp$ and let $\Fil^i S_{K_0}:= \Fil^i S\otimes_{\Zp}\Qp$.
Let $\bigMF$ be the category whose objects are finite free $S_{K_0}$-modules $\D$ with:
\begin{enumerate}
 \item a $\varphi_{S_{K_0}}$-semi-linear morphism $\varphi_{\D}: \D \to \D$ such that the determinant of $\varphi_{\D}$ is invertible in $S_{K_0}$;
 \item a decreasing filtration $\{\Fil^i\D\}_{i \in \Z}$ of $S_{K_0}$-submodules of $\D$ such that $\Fil^0\D=\D$ and $\Fil^i S_{K_0} \Fil^j \D \subseteq \Fil^{i+j}\D$;
 \item a $K_0$-linear map $N: \D \to \D$ such that $N(fm)=N(f)m+fN(m)$ for all $f\in S_{K_0}$ and $m \in \D$, $N\varphi=p \varphi N$ and $N (\Fil^i \D) \subseteq \Fil^{i-1}\D$.
\end{enumerate}

Morphisms in the category are $S_{K_0}$-linear maps preserving filtrations and commuting with $\varphi$ and $N$. A sequence $0 \to \D_1 \to \D \to
\D_2 \to 0$ is called short exact if it is short exact as $S_{K_0}$-modules and the sequences on filtrations $0 \to \Fil^i \D_1 \to \Fil^i \D \to
\Fil^i \D_2 \to 0$ are short exact for all $i$. We call $\D_2$ a quotient of $\D$ in this case.

Following \cite{Bre97}, let $\Asthat$ be the $p$-adic completion of the PD polynomial algebra $\Acris\langle X \rangle$. We extend the natural $G_K$-action
and Frobenius on $\Acris$ to $\Asthat$ by $g(X)=\underline{\epsilon}(g)X +\underline{\epsilon}(g) -1 $, where $\underline
{\epsilon}(g)=\frac{g([\underline{\pi}])}{[\underline{\pi}]}$, and $\varphi(X)=(1+X)^p-1$. And we define a monodromy operator
$N$ on $\Asthat$ to be the unique $\Acris$-linear derivation such that $N(X)=1+X$. For any $i \geq 0$, we define
$$\Fil^i \Asthat = \{ \sum_{j=0}^{\infty} a_j \gamma_j(X), a_j \in \Acris, \lim_{j\to \infty}a_j=0, a_j \in \Fil^{i-j}\Acris \text{ for } 0 \leq j \leq i   \}.$$
$\Asthat$ is an $S$-algebra by $u \mapsto [\underline{\pi}](1+X)^{-1}$. Define $ V_{\textnormal{st}}(\mathcal D) := \Hom_{S, \Fil, \varphi, N} (\bigD, \Asthat[1/p])$.

For $D \in \MF$, we can associate an object in $\bigMF$ by $\bigD:= S\otimes_{W(k)}D$ and
 \begin{enumerate}
 \item $\varphi: =\varphi_S \otimes \varphi_D$;
 \item $N:= N\otimes Id + Id\otimes N$;
 \item $\Fil^0\bigD :=\bigD$ and inductively,
 $$\Fil^{i+1}\bigD := \{ x \in \bigD | N(x) \in \Fil^i \bigD \text{ and } f_{\pi}(x) \in \Fil^{i+1}D_K  \},$$
 where $f_{\pi}: \bigD \twoheadrightarrow D_K$ by $s(u)\otimes x \mapsto s(\pi)x$.
 \end{enumerate}

\begin{thm}[\cite{Bre97}] \label{bigD}
The functor above induces an equivalence between $\MF$ and $\bigMF$, and it is compatible with Galois representations, i.e., $V_{\textnormal{st}}(D) \simeq V_{\textnormal{st}}(\mathcal{D})$ as $\Qp[G]$-modules.
\end{thm}
\begin{rem}
We will always identify $V_{\textnormal{st}}(D)$ with $V_{\textnormal{st}}(\D)$ as the same Galois representations, and we denote $\bigMFwa$ as the essential image of the functor $\D$ restricted to $\MFwa$. Also, the functor $\D$ is exact.
\end{rem}

\begin{defn}
Let $D \in \MFwa$, $\bigD =\bigD(D)$. A quasi-strongly divisible lattice of weight $r$ in $\bigD$ is an $S$-submodule
$\bigM$ of $\bigD$ such that,
\begin{enumerate}
\item  $\M$ is $S$-finite free and $\M[1/p] \simeq \D$;
\item  $\M$ is stable under $\varphi$, i.e., $\varphi(\M) \subseteq \M$;
\item $\varphi(\Fil^r \M) \subseteq p^r\M$ where $\Fil^r \M := \M \cap \Fil^r \D$.
\end{enumerate}
A strongly divisible lattice of weight $r$ is a quasi-strongly divisible lattice $\M$ which is stable under monodromy, i.e., $N(\M) \subseteq \M$.
\end{defn}

Let $'\Mod_{S}^{\varphi, N}$ be the category whose objects are 4-tuples $(\M, \Fil^r\M, \varphi_r, N)$, where
\begin{enumerate}
\item $\M$ is an $S$-module, $\Fil^r \M \subseteq \M$ is an $S$-submodule which contains $\Fil^r S \M$;
\item $\varphi_r : \Fil^r \M \to \M$ is a $\varphi_{S}$-semi-linear morphism such that for all $s \in \Fil^r S$ and $x \in \M$, we have $\varphi_r(sx)=c^{-r}\varphi_r(s)\varphi_r(E(u)^rx)$;
\item $N: \M \to \M$ is a $W(k)$-linear map such that $N(sx)=N(s)x+sN(x)$ for all $s \in S, x \in \M$, $E(u)N(\Fil^r \M) \subseteq \Fil^r \M$ and the following
diagram commutes:
$$\xymatrix{\Fil^r \M \ar[d]^{E(u)N} \ar[r]^{\varphi_r} & \M \ar[d]^{cN}\\ \Fil^r \M \ar[r]^{\varphi_r} & \M} $$
\end{enumerate}
Morphisms in the category are given by $S$-linear maps preserving $\Fil^r$ and commuting with $\varphi_r$ and $N$.

\begin{rem}
If we forget the $N$-structure in the above definition, this is precisely $'\Mod_{S}^{\varphi}$ as we defined in Section 2.2.
\end{rem}

Let $\ModFI_{S}^{\varphi, N}$ be the full subcategory of $'\Mod_{S}^{\varphi, N}$ consisting of objects such that:
\begin{enumerate}
\item as an $S$-module, $\M \simeq \oplus_{i \in I}S_{n_i}$, where $I$ is a finite set;
\item $\varphi_r(\Fil^r \M)$ generates $\M$ over $S$.
\end{enumerate}

Finally, we denote $\Mod_{S}^{\varphi, N}$ the full subcategory of $'\Mod_{S}^{\varphi, N}$ such that $\M$ is a finite free $S$-module and for all $n$,
$(\M_n, \Fil^r \M_n, \varphi_r, N) \in \ModFI_{S}^{\varphi, N}$, here $\M_n =\M/p^n$. A sequence in $\Mod_{S}^{\varphi, N}$ is said to be short exact if it is short exact as a sequence in $\Mod_{S}^{\varphi}$.

We can show that $\Asthat \in '\Mod_{S}^{\varphi, N}$. For $\M \in \Mod_{S}^{\varphi, N}$ of $S$-rank $d$,
define $$T_{\textnormal{st}}(\M):= \Hom_{'\Mod_{S}^{\varphi, N}}(\M, \Asthat)$$ as in Section 2.3.1 of \cite{Bre99a}, it is a finite free $\Zp$-representation of $G_K$ of rank $d$.

\begin{prop}[Breuil] \label{strong}
\begin{enumerate}
\item If $\M$ is a quasi-strongly divisible lattice in $\D \in \bigMFwa$, then $(\M, \Fil^r\M, \varphi_r)$ is in $Mod_S^{\varphi}$, where $\varphi_r := \frac{\varphi}{p^r}$.
\item The category of strongly divisible lattices of weight $r$ is equivalent to $\Mod_{S}^{\varphi, N}$. In particular,
for $\M \in \Mod_{S}^{\varphi, N}$, there exists $D \in \MFwa$, such that $\bigD(D) \simeq \M \otimes_{\Zp} \Qp$ as filtered
$(\varphi, N)$-modules over $S$. Furthermore, $T_{\textnormal{st}}(\M)$ is a $G$-stable $\Zp$-lattice in $V_{\textnormal{st}}(D)$.
\end{enumerate}
\end{prop}
\begin{proof}
This is Theorem 2.2.3 in \cite{Bre02}.
\end{proof}
\begin{rem}
From now on, we will use $\Mod_{S}^{\varphi, N}$ to denote the category of strongly divisible lattices of weight $r$. And we use $\widetilde{\Mod_S^{\varphi}}$ to denote the category of quasi-strongly divisible lattices.
\end{rem}

\subsection{Unipotency}

\begin{defn}
$\D \in \bigMF$ (resp. $\bigMFwa$) is called \'{e}tale, multiplicative, nilpotent or unipotent if its corresponding $D \in \MF$ (resp. $\MFwa$) is so.
\end{defn}

\begin{prop}
$\D$ is \'{e}tale if and only if $\Fil^r \D =\D$. $\D$ is multiplicative if and only if $\Fil^i \D = \Fil^i S_{K_0} \D$ for some $1 \leq i \leq r$ (equivalently, for all $1 \leq i \leq r$).
\end{prop}

\begin{proof}
Let $D \in \MF$ be the module corresponding to $\D$.
Note that by Theorem 6.1.1 of \cite{Bre97}, $D \simeq \D/u\D$, and the filtration on $D_K$ is induced from $D_K =\D/(\Fil^1S_{K_0} \D)$, i.e., $\Fil^i D_K = (\Fil^i \D + \Fil^1S_{K_0} \D) / (\Fil^1S_{K_0} \D)$.

If $D$ is \'{e}tale, i.e., $\Fil^r D_K =D_K$, from the definition of filtration of $\D(D)$, inductively it can be deduced that $\Fil^r \D =\D$. On the other hand, if $\Fil^r \D =\D$, then $\Fil^r D_K = (\Fil^r \D + \Fil^1S_{K_0} \D) / (\Fil^1S_{K_0} \D) = D_K$.

If $D$ is multiplicative, i.e., $\Fil^1 D_K=0$, then from the definition of filtration of $\D(D)$, inductively it can be deduced that $\Fil^i \D = \Fil^i S_{K_0} \D$ for all $1 \leq i \leq r$ (note that in this case $N(D)=0$).
On the other hand, if $\Fil^i \D = \Fil^i S_{K_0} \D$ for some $1 \leq i \leq r$, we claim that $\Fil^1 \D = \Fil^1 S_{K_0} \D$. To prove the claim, we assume $i>1$, note that we always have
$\Fil^1 \D \supseteq \Fil^1 S_{K_0} \D$. Now given any $x \in \Fil^1 \D$, then $E(u)^{i-1}x \in \Fil^{i-1}S_{K_0}\Fil^1 \D \subseteq \Fil^i \D =\Fil^i S_{K_0}\D$. It is easy to show that $E(u)^{i-1}S_{K_0} \cap \Fil^i S_{K_0}= E(u)^{i-1}\Fil^1 S_{K_0}$, thus $x \in \Fil^1 S_{K_0}\D$. This proves the claim. So $\Fil^1 D_K=0$, and inductively, $\Fil^i \D = \Fil^i S_{K_0} \D$ for all $1 \leq i \leq r$.
\end{proof}

\begin{thm} \label{etalequasi}
Let $V$ be a semi-stable Galois representation, $D=D_{\textnormal{st}}(V) \in \MFwa$, and $\D \in \bigMFwa$ the module corresponding to $D$. Suppose $\M \subset \D$ is a quasi-strongly divisible lattice. Then $V$ is \'{e}tale (or multiplicative, nilpotent, unipotent) if and only if $\bigM$ (regarded as a module in $\Mod_{S}^{\varphi}$ via Proposition \ref{strong}) is \'{e}tale (or multiplicative, nilpotent, unipotent).
\end{thm}

\begin{proof}
It is easy to check that $\M^{\vee}$ is a quasi-strongly divisible lattice of $\D^{\vee}$, thus by duality, we only need to prove the theorem for \'{e}tale and nilpotent representations.

Suppose $V$ is \'{e}tale, then $\bigD$ is \'{e}tale, so $\Fil^r \bigD = \bigD$. Thus $\Fil^r \M = \M \cap \Fil^r \D = \M$, i.e., $\M$ is \'{e}tale.

Suppose $V$ is nilpotent, if $\M$ is not nilpotent, then $\M^{\textnormal{m}}$ is nonzero. Let $\D^{\textnormal{m}} = \M^{\textnormal{m}}\otimes
S[1/p]$ with a filtration induced from that of $\D$. We claim that $\D^{\textnormal{m}}$ is a submodule of $\D$ in the category $\bigMFwa$.

$\D^{\rm m}$ is $\varphi$-stable because $\M^{\rm m}$ is. Next we show that $\D^{\rm{m}}$ is $N$-stable. In fact, we can show that $\M^{\rm{m}}$ is $N$-stable. Since $\M \otimes S[1/p]=\D$ is $N$-stable, thus there exists $k \in \N$, such that $p^k N(\M) \subseteq \M$. Let $x \in \M$. By using the relation $N \varphi = p\varphi N$ in $\D$, we have $N \varphi^n(x) = \varphi^n(p^n N(x))$. Thus for $n \geq k$, $(\varphi^{\ast})^n(\M)$ is $N$-stable. By Proposition \ref{bigMm}, $\M^{\textnormal{m}} = \cap (\varphi^{\ast})^n(\M)$, so $\M^{\rm m}$ is also $N$-stable. Thus, $\D^{\rm{m}} \in \bigMF$.

The only thing left now is to show that $\D^{\textnormal{m}}$ is weakly admissible. Suppose $D^{\textnormal{m}} \in \MF$ is the corresponding module, then we have $t_H(D^{\textnormal{m}}) = t_N(D^{\textnormal{m}})$ by Proposition 2.1.3 of \cite{Bre99a} (since $\M^{\textnormal{m}}$ is a quasi-strongly divisible lattice in $\D^{\textnormal{m}}$). For any submodule $D'$ of $D^{\textnormal{m}}$, since it is also a submodule of the weakly admissible module $D$, $t_H(D') \leq t_N(D')$. Thus $D^{\textnormal{m}}$ is weakly admissible, and so is $\D^{\textnormal{m}}$.

Now we have proved that $\D^{\textnormal{m}}$ is a submodule of $\D$ in the category $\bigMFwa$, and since $\Fil^r \D^{\textnormal{m}} = \D^{\textnormal{m}} \cap \Fil^r \D = \M^{\textnormal{m}}\otimes S[1/p]\cap (\Fil^r\M)\otimes S[1/p]= (\Fil^rS \M^{\textnormal{m}})\otimes S[1/p] = \Fil^r S_{K_0}\D^{\textnormal{m}}$, so $\D^{\textnormal{m}}$ is a multiplicative object. This contradicts the nilpotency of $\D$.

For the other direction, if $\M$ is \'{e}tale, then $\Fil^r \M =\M$, so $\M \subset \Fil^r \D$. Thus $\Fil^r \D =\D$, i.e., $\D$ is \'{e}tale. If $\M$ is nilpotent and $\D$ is not, then there is a nonzero multiplicative $\D^{\textnormal{m}}$ which will contain a nonzero multiplicative quasi-strongly divisible lattice $\M'$, and for $n$ big enough, $p^n \M' \subseteq \M$. Since $\M'$ is multiplicative, $p^n \M' \subseteq \M^{\textnormal{m}}$, contradicting the nilpotency of $\M$.
\end{proof}

\begin{rem}
In the case $e=1$, $N_D=0$ (i.e., crystalline representations), if we tensor the $W(k)$-lattice constructed in \cite{FL82} by $S$, it will give us an example of strongly divisible lattice (see Example 2.2.2 (2) of \cite{Bre02}). It is easy to check that the notion of ``unipotency" of $W(k)$-lattices in \cite{FL82} is compatible with ours, i.e., unipotent $W(k)$-lattices will give us unipotent $S$-lattices.
\end{rem}

Now we can state our main theorem.
\begin{thm}[Main Theorem]\label{mainthm}
Let $p$ be a prime, $r =p-1$. The functor $T_{\textnormal{st}}: \Mod_S^{\varphi, N} \to \Rep_{\Zp}^{\textnormal{st}}(G_K)$ establishes an anti-equivalence between the category of unipotent strongly divisible lattices of weight $r$ and the category of $G_K$-stable $\Zp$-lattices in unipotent semi-stable $p$-adic Galois representations of $G_K$ with Hodge-Tate weights in $\{0, \ldots, r\}$.
\end{thm}

\section{Proof of the main theorem}

\subsection{Full faithfulness of $T_{\textnormal{st}}$}

Let $\Mod_{\huaS}^{\varphi, N}$ be the category whose objects are triples $(\huaM, \varphi, N)$,
 where $(\huaM, \varphi)$ is an object in $\Mod_{\huaS}^{\varphi}$,
  and $N : \huaM/u\huaM\otimes_{\Zp} \Qp \to \huaM/u\huaM\otimes_{\Zp} \Qp $ is a linear homomorphism such that $N\varphi =p\varphi N$. The morphisms are $\huaS$-linear maps that are compatible with $\varphi$ and $N$. The following theorem is a main result in \cite{Kis06}.

\begin{thm}[Kisin] \label{Dtom}
  There is a fully faithful $\otimes$-functor $\theta : \MFwa \to \Mod_{\huaS}^{\varphi, N}$, and there is a canonical bijection  $V_{\textnormal{st}}(D)   \simeq T_{\huaS}({\theta(D)}) \otimes_{\Zp}\Qp$ which is compatible with $G_{\infty}$-actions.
\end{thm}



\begin{prop}[Kisin] \label{ThuaS}
\begin{enumerate}
\item For any $G_{\infty}$-stable $\Zp$-lattice $T$ in a semi-stable Galois representation $V$, there always exists an $\huaN \in \Mod_{\huaS}^{\varphi}$ such that $T_{\huaS}(\huaN) \simeq T$.
\item $T_{\huaS}: \Mod_{\huaS}^{\varphi} \to \Rep_{\Zp}(G_{\infty})$ is fully faithful.
\end{enumerate}
\end{prop}
\begin{proof}
This is Proposition 2.1.12 and Lemma 2.1.15 of \cite{Kis06}. Remark that the Proposition is valid for any semi-stable representation with Hodge-Tate weights $\subseteq \{0, \ldots, r\}$ for any $0 \leq r < \infty$.
\end{proof}


\begin{prop} \label{Tcris}
Let $\widetilde{\Mod_S^{\varphi,\rm u}}$ be the category of unipotent quasi-strongly divisible lattices of weight $p-1$, $\Rep_{\Zp}^{\rm st,\rm u}(G_{\infty})$ the category of $G_{\infty}$-stable $\Zp$-lattices in unipotent semi-stable representations with Hodge-Tate weights $\subseteq \{0, \ldots, p-1\}$. Then $T_{\rm cris}: \widetilde{\Mod_S^{\varphi,\rm u}} \to \Rep_{\Zp}^{\rm st,\rm u}(G_{\infty})$ establishes an anti-equivalence.
\end{prop}

\begin{proof}
The proof of essential surjectivity of $T_{\rm cris}$ is almost verbatim as in Proposition 3.4.6 of \cite{Liu08}. Since the proof itself is useful in our following discussion, we will give a sketch (for full proof, see \cite{Liu08}): Let $V$ be a unipotent semi-stable Galois representation, $D=D_{\textnormal{st}}(V) \in \MFwa$, and $\D \in \bigMFwa$ the module corresponding to $D$. By Theorem \ref{Dtom}, let $\huaM =\theta(D)$, then $S_{K_0} \otimes_{\varphi, \huaS}\huaM \simeq \D$ by Corollary 3.2.3 of \cite{Liu08}. Given a $G_{\infty}$-stable $\Zp$-lattice $T$ in $V$, by Proposition \ref{ThuaS}, there exists $\huaN$ such that $T_{\huaS}(\huaN) \simeq T$. We can prove that $\huaM \otimes_{\Zp}\Qp \simeq \huaN \otimes_{\Zp}\Qp$. Now put $\bigN = \bigM_{\huaS}(\huaN)$, then $\bigN$ is a quasi-strongly divisible lattice in $\bigN\otimes_{\Zp}\Qp \simeq \M\otimes_{\Zp}\Qp=\D$. By Theorem \ref{etalequasi}, $\bigN$ is unipotent, so by Theorem \ref{mtoM}, $T_{\text{cris}}(\bigN) \simeq T_{\huaS}(\huaN) \simeq T$.

To prove full faithfulness, let $\bigM, \bigN \in \widetilde{\Mod_S^{\varphi,\rm u}}$, and suppose that there is a morphism $f: T_{\text{cris}}(\bigN) \to T_{\text{cris}}(\bigM)$. Since the representations are unipotent, by Theorem \ref{etalequasi}, both $\bigM$ and $\bigN$ are unipotent $S$-modules. By Theorem \ref{mtoM}, there are $\huaM, \huaN \in \Mod_{\huaS}^{\varphi}$ such that $\bigM = \bigM_{\huaS}(\huaM), \bigN=\bigM_{\huaS}(\huaN)$, and $T_{\huaS}(\huaM) =T_{\text{cris}}(\bigM), T_{\huaS}(\huaN) =T_{\text{cris}}(\bigN)$. Thus we have a morphism $f: T_{\huaS}(\huaN) \to T_{\huaS}(\huaM)$. By Proposition \ref{ThuaS}, there exists a map $\mathfrak{f} : \huaM \to \huaN$ which induces $f$. By Theorem \ref{mtoM}, $\bigM_{\huaS}(\mathfrak{f}): \bigM \to \bigN$ induces $f$ as well.
\end{proof}

\begin{rem}
Note that we actually do not need Lemma 3.4.7 in \cite{Liu08} to prove the full faithfulness.
\end{rem}

\begin{prop}\label{Tstff}
The functor $T_{\textnormal{st}}$ in Theorem \ref{mainthm} is fully faithful.
\end{prop}
\begin{proof}
This is the $r=p-1$ unipotent case generalization of Corollary 3.5.2 of \cite{Liu08}. The proof is verbatim as in  \cite{Liu08} because Proposition \ref{Tcris} is the $r=p-1$ unipotent case generalizations of Proposition 3.4.6 of  \cite{Liu08}.
\end{proof}

\subsection{Stability of monodromy and essential surjectivity of $T_{\textnormal{st}}$}

Now, for any $G_K$-stable $\Zp$-lattice in a unipotent semi-stable Galois representation $V$, by Proposition \ref{Tcris}, there exists a quasi-strongly divisible lattice $\bigM$ in $\bigD$ such that $T_{\textnormal{cris}}(\bigM) = T|_{G_{\infty}}$. By Proposition 3.5.1 of \cite{Liu08}, if $N(\bigM) \subseteq \bigM$, then $(\bigM, \Fil^r \bigM, \varphi, N)$ is a strongly divisible lattice in $\bigD$ and $T_{\textnormal{st}}(\bigM)=T$. Thus in order to prove the essential surjectivity of $T_{\textnormal{st}}$, it remains to show that $\bigM$ is also $N$-stable, i.e., $N(\M) \subseteq \M$.

When $p>2$, Proposition 2.4.1 of \cite{Liu12} will prove this ``monodromy stability", hence the essential surjectivity of $T_{\textnormal{st}}$. But for $p=2$, we need a separate treatment by using the strategy in \cite{Bre02}.

\subsubsection{The case $p>2$}

Here, we give a strengthening (Theorem \ref{NSigma}) of the stability result proved in \cite{Liu12}, showing that the coefficients of the monodromy operator can in fact be put in $\Sigma$. We will first give a brief sketch of the the proof in \cite{Liu12}, and only introduce notations that are useful to prove our result.

Let
$$I^{[n]}B_{\text{cris}}^{+} := \{ x \in  B_{\text{cris}}^{+}: \varphi^k(x) \in \Fil^n B_{\text{cris}}^{+}, \text{ for all } k>0  \},$$
and for any subring $A \subseteq B_{\text{cris}}^{+}$, write $I^{[n]}A:= A\cap I^{[n]}B_{\text{cris}}^{+}$. By the argument preceding Theorem 2.2.1 of \cite{Liu12}, there exists a nonzero $\huat \in W(R)$ such that $t= \lambda \varphi(\huat)$ with $\lambda = \Pi_{k=1}^{\infty} \varphi^k(\frac{c_0^{-1}E(u)}{p}) \in S^{\times}$ where $c_0 =\frac{E(0)}{p}$, here $t = \log([\underline{\epsilon}])$. For all $n$, $I^{[n]}W(R)$ is a principle ideal (cf. Proposition 6.18 of \cite{FO}), and by the proof of Lemma 3.2.2 in \cite{Liu10}, $(\varphi(\huat))^n$ is a generator of $I^{[n]}W(R)$.

By Lemma 5.1.1 of \cite{Liu08}, there is a $G$-action on $B_{\text{cris}}^{+} \otimes_{S_{K_0}} \D$ defined by,
$$ \sigma(a\otimes x) = \sum_{i=0}^{\infty} \sigma(a) \gamma_{i}(-\log(\underline{\epsilon}(\sigma))) \otimes N^i (x), \forall \sigma \in G, a\in B_{\text{cris}}^{+}, x \in \D,$$
where $\underline{\epsilon}(\sigma) = \frac{\sigma([\underline{\pi}])}{[\underline{\pi}]}$. By Lemma 2.3.1 of \cite{Liu12},  $W(R) \otimes_S \M$ is $G$-stable with the induced $G$-action on $B_{\text{cris}}^{+} \otimes_{S_{K_0}} \D$.

\begin{proof}[Proof of the Main Theorem \ref{mainthm} for $p>2$.]

Now for our $G_{K}$-stable $\Zp$-lattice $T$, by Proposition \ref{Tcris}, there
exists $\M$ such that $T_{\text{cris}}(\M)=T|_{G_{\infty}}$. Since we are in a unipotent
representation, there exists a unipotent $\huaM$ such that $\M =S\otimes_{\varphi, \huaS}\huaM$,
and $T_{\huaS}(\huaM) =T_{\text{cris}}(\M)$. In order to prove $\M$ is $N$-stable, it suffices to
prove that $N(\huaM) \subseteq \M$ (regard $\huaM$ as a $\varphi(\huaS)$-submodule of $\M$).
Now the proof goes the same as Proposition 2.4.1 of \cite{Liu12}. Since the proof will be used in Theorem \ref{NSigma}, we give a brief sketch.

Recall that we can choose $\tau$ a topological generator of $G_{p^{\infty}}$ such that  $-\log([\underline{\epsilon}(\tau)]) = t$ (Note that the existence of $\tau$ relies on the fact $p>2$). And it is easily shown that for $x \in \huaM$,
$$ (\tau -1)^n(x) = \sum_{m=n}^{\infty} ( \sum_{i_1+\ldots+i_n=m, i_j \geq 1} \frac{m!}{i_1!\ldots i_n!} ) \gamma_m(t)\otimes N^m(x) .$$

Now $(\tau -1)^n(x) \in I^{[n]}B_{\text{cris}}^{+}\otimes_S \D$, and since $W(R) \otimes_{\varphi, \huaS} \huaM = W(R) \otimes_S \M$ is $G$-stable, $(\tau -1)^n(x) \in I^{[n]}W(R)\otimes_{\varphi, \huaS}\huaM$.
We then can show that $\frac{(\tau -1)^n}{nt}(x)$ is well defined in $A_{\text{cris}}\otimes_S \M$ for all $n$, and  $\frac{(\tau -1)^n}{nt}(x) \to 0$. Then we have $$1\otimes N(x)= \sum_{n=1}^{\infty} (-1)^{n-1} \frac{(\tau -1)^n}{nt}(x),$$ and the infinite sum converges in $A_{\text{cris}}\otimes_S \M$, so $N(x) \in \M$.

\end{proof}

With the above result proved, we can in fact show that the coefficients of $N$ can be put into $\Sigma$.

\begin{thm} \label{NSigma}
Suppose $p >2$, then we can choose a basis of $\M$ such that the matrix of $N$ with respect to the basis is in $\Mat_d(\Sigma)$. Indeed, we have $N(\huaM) \subset \Sigma \otimes_{\varphi, \huaS} \huaM$.
\end{thm}

We start with two lemmas.

\begin{lemma} \label{Liulemma}
$x \in A_{\textnormal{cris}}$, $j \geq i \geq 0$, $E(u)^i x \in \Fil^j \Acris$, then $x \in \Fil^{j-i}\Acris$.
\end{lemma}

\begin{proof}
This is Lemma 3.2.2 of \cite{Liu08}.
\end{proof}

\begin{lemma} \label{LemmaEU}
Let $a \in K_0[u], a \neq 0, \deg(a)<e$. If $aE(u)^n \in W(R) + \Fil^{n+1}\Acris$ for some $n \geq 0$, then $a \in W(k)[u]$.
\end{lemma}
\begin{proof}
Suppose otherwise, $a = \frac{b}{p^s}$, where $s \geq 1$, and $b \in W(k)[u]$ of degree less than $e$ and at least one of the coefficients of $b$ is a unit in $W(k)$. Thus $v_p(b(\pi))<1$. Suppose $aE(u)^n = x+y$, where $x \in W(R), y \in \Fil^{n+1}\Acris$. If $n=0$, then $b-p^sx=p^sy$. Apply the homomorphism $\theta: \Acris \to \mathcal{O}_C$, since $p^sy$ maps to 0, $b(\pi)=\theta(b)=p^s\theta(x)$, $v_p(b(\pi)) \geq s \geq 1$, contradiction. If $n >0$, then apply the homomorphism $\theta$ to $bE(u)^n =p^sx+p^sy$, so $\theta(x)=0$. Since $\textnormal{Ker}(\theta \mid_{W(R)})$ is principle, so $x=E(u)x_1$ with $x_1 \in W(R)$, and $E(u)(bE(u)^{n-1} - p^sx_1) \in \Fil^{n+1}\Acris$. Since $bE(u)^{n-1} - p^sx_1 \in \Acris$,  by Lemma \ref{Liulemma}, $bE(u)^{n-1} - p^sx_1 \in \Fil^{n}\Acris$. Now an induction on $n$ will show us the result.
\end{proof}

\newcommand{\floor}[1]{\lfloor{#1}\rfloor}

\begin{proof}[Proof of Theorem \ref{NSigma}]

Fix a $\huaS$-basis $(e_1, \ldots, e_d)$ of $\huaM$. For any $x \in \huaM$ as in the proof of Theorem \ref{mainthm} for $p>2$, we have
$$1\otimes N(x)= \sum_{n=1}^{\infty} (-1)^{n-1} \frac{(\tau -1)^n}{nt}(x) =  \sum_{n=1}^{\infty} (\sum_{i=1}^d a_{n,i}(x)e_i).$$

Here, $a_{n,i}(x) \in A_{\rm cris}$ depends on $x$, but in the following, we simply use $a_{n,i}$ to denote them.

Since $(\tau -1)^n(x) \in I^{[n]}W(R)\otimes_{\varphi, \huaS}\huaM$, if we multiply $1\otimes N(x)$ by $\lambda p^k$ ($k \geq 1$), then the first $p^k$ terms in the sum (i.e., $\lambda p^k (-1)^{n-1} \frac{(\tau -1)^n}{nt}(x)$ for $1 \leq n \leq p^k$) will have coefficients in $W(R)$, namely, $\lambda p^ka_{n,i} \in W(R), 1 \leq n \leq p^k$. In fact, the $(p^k+1)$-th term will also have coefficients in $W(R)$ since $(p^k+1)$ in the denominator is a unit, i.e., $\lambda p^ka_{p^k+1,i} \in W(R)$.
For the remaining terms, $\sum_{n=p^k+2}^{\infty} (-1)^{n-1} \frac{(\tau -1)^n}{nt}(x) \in \Fil^{p^k+1}B_{\textnormal{cris}}^{+} \otimes_{\varphi, \huaS}\huaM$ because $t^n \in \Fil^n B_{\textnormal{cris}}^{+}$. But it is also in $A_{\textnormal{cris}}\otimes_{\varphi, \huaS}\huaM$ by the proof of Theorem \ref{mainthm}, so these terms have coefficients in $\Fil^{p^k+1} A_{\rm{cris}}$, i.e, $\lambda p^ka_{n,i} \in \Fil^{p^k+1} A_{\rm{cris}}, n \geq p^k+2$.

Thus for any entry $y$ in the matrix of $N$, we conclude

\begin{equation}
 \lambda p^ky \in W(R)+ \Fil^{p^k+1} A_{\textnormal{cris}}, \forall k \geq 1. \label{seq} \tag{$\ast$}
\end{equation}

Note that we already have $\lambda y \in S$, and in the following we will show that $\lambda y \in \Sigma$.

Let $\lambda y = \displaystyle\sum_{i=0}^{\infty} a_i \frac{E(u)^i}{i!}$ as an element of $S$, where $a_i \in W(k)[u]$ of degree less than $e$ and $v_p(a_i) \to 0$. We claim that we have $\frac{p^ka_i}{i!} \in W(k)[u], \forall i \leq p^k$.
To prove the claim, fix any $k$, and do an induction on $i$. It is clearly true for $i=0$. Suppose it is true for $ \leq i$. Now for $i+1$, by (\ref{seq}), we have (note that $i+2 \leq p^k+1$),

$$\frac{p^ka_{i+1}}{(i+1)!}E(u)^{i+1} =p^k\lambda y - p^k(\sum_{j=0}^{i} a_j \frac{E(u)^j}{j!} + \sum_{j=i+2}^{\infty} a_j \frac{E(u)^j}{j!})  \in W(R) + \Fil^{i+2} A_{\text{cris}},$$
now apply Lemma \ref{LemmaEU} and the claim is proved.

So we have

$$\lambda y = \sum_{i=0}^{p^2-1}a_i \frac{E(u)^i}{i!}  +  a_{p^2} \frac{E(u)^{p^2}}{p^2!} +  \sum_{k=3}^{\infty} \sum_{i=p^{k-1}+1}^{p^k} a_i \frac{E(u)^i}{i!} .$$

It is easy to verify that $\displaystyle\sum_{i=0}^{p^2-1}a_i\frac{E(u)^i}{i!} \in \Sigma$.
And since $\frac{p^2a_{p^2}}{p^2!} \in W(k)[u]$ by what we proved above,  $a_{p^2} \frac{E(u)^{p^2}}{p^2!} \in \Sigma$.

For the third part in the sum, for any $p^{k-1}+1 \leq i \leq p^k$,
$$a_i \frac{E(u)^i}{i!} =  b_i \frac{E(u)^i}{p^k} = c_i (\frac{E(u)^p}{p})^{\floor{\frac{i}{p}}} E(u)^{i-p\floor{\frac{i}{p}}},  $$

here $b_i= \frac{p^ka_i}{i!} \in W(k)[u]$, and $c_i =b_i p^{\floor{\frac{i}{p}}-k} \in W(k)[u]$ because $\floor{\frac{i}{p}}-k \geq 0, \forall k \geq 3, i >p^{k-1}$. Clearly $v_p(c_i) \to 0$, thus the third part is also in $\Sigma$.

All three parts are in  $\Sigma$, thus $\lambda y \in \Sigma$, so $y \in \Sigma$ because $\lambda$ is a unit in $\Sigma$, and we are done.
\end{proof}

\subsubsection{The case $p=2$}

For $p=2$, we will utilize the strategy in Subsection 3.5 of \cite{Bre02}. We first collect some results from \cite{Bre02}.

\begin{lemma} \label{collect}
Let $p$ be a prime. Let $(\rho, V, D)$ be a semi-stable representation with Hodge-Tate weights in $\{0, 1\}$, where $D=D_{\textnormal{st}}(V)$. Let $D'$ be the filtered $\varphi$-module by forgetting the $N$-structure in $D$. $\D = S \otimes D$, $\D' = S\otimes D'$. Let $\M \subset \D$ be a strongly divisible lattice. Let $\M'$ be the image of $\M$ under the identification $\D \simeq \D'$.
\begin{enumerate}
\item $D'$ is also weakly admissible.
\item $V_{\textnormal{st}}(D') \simeq V_{\textnormal{st}}(D)$ as vector spaces.
\item $\M'$ is stable under $N_{\D'} = N_S \otimes 1$, thus $\M'$ is a strongly divisible lattice in $\D'$.
\item Under the identification $V_{\textnormal{st}}(D') \simeq V_{\textnormal{st}}(D)$, the lattice $ T'= T_{\textnormal{st}}(\M')$ corresponds to $T=T_{\textnormal{st}}(\M)$.
\end{enumerate}
\end{lemma}
\begin{proof}
These are summarized from Lemma 3.5.1, Lemma 3.5.3 and Lemma 3.5.6 of \cite{Bre02}. They were proved for $p>2$, but they are still true for $p=2$ (and for all semi-stable representations, not just unipotent ones).
\end{proof}

Now let $(\rho, V, D)$ be a \emph{unipotent} semi-stable representation, and in the following we will use notations from the above lemma, i.e., $(\rho, V, D, T, \D, \M)$ and $(\rho', V', D', T', \D', \M')$.
\begin{prop}
$D'$ is also unipotent.
\end{prop}
\begin{proof}
Since $D$ is unipotent, by Theorem \ref{etalequasi},  $\M$ is unipotent. Note that $\Fil^1\D' = \Fil^1 \D$, thus $\Fil^1 \M' =\Fil^1 \M$. Since the $\varphi$-structures on $\D'$ and $\D$ are the same,  $\M'$ is a quasi-strongly divisible lattice in $\D'$. And $\M'$ is unipotent because it has the same $\Fil^1$ and $\varphi$-structure as $\M$. Thus $D'$ is unipotent.
\end{proof}

\begin{prop} \label{p2crysff}
Let $p=2$, then $T_{\textnormal{st}}$ induces an anti-equivalence of categories between the category of unipotent strongly divisible lattices $\M$ of weight $1$ such that $N(\M) \subset u\M$ and the category of $G_K$-stable $\Zp$-lattices in unipotent crystalline representations of $G_K$ with Hodge-Tate weights in $\{0, 1\}$.
\end{prop}

\begin{proof}
This is the generalization of Theorem 3.2.5 in \cite{Bre02}. We have already proved the full faithfulness in Proposition \ref{Tstff}. For the essential surjectivity, by the argument at the beginning of Subsection 4.2, we need to show that the quasi-strongly divisible lattice $\M$ we constructed there is $N$-stable, and this is part (1) of Lemma 3.5.6 of \cite{Bre02}. In fact, $\M$ is not only $N$-stable, we even have $N(\M) \subset u\M$.
\end{proof}

\begin{proof} [Proof of the Main Theorem \ref{mainthm} for $p=2$.]

Let $D_0 = N(D)$ be the image of the monodromy operator on $D$, by the proof of Lemma 3.5.3 in \cite{Bre02}, $D_0$ is weakly admissible and $\Fil^1 (D_0\otimes K) : = (D_{0}\otimes K) \cap \Fil^1 D_K =0$, so $D_0$ is multiplicative. Consider the natural embedding $D_0 \hookrightarrow D'$, it is an injective morphism in the category of filtered $\varphi$-modules, so it induces a surjection of representations $V' \twoheadrightarrow V_0$. Let $T_0$ be the image of $T'$ under the surjection, then $T_0$ is a $G_K$-stable $\Zp$-lattice in $V_0$. Since both $D'$ and $D_0$ are unipotent ($D_0$ is multiplicative, thus unipotent), by Proposition \ref{p2crysff}, the morphism $T' \to T_0$ induces a morphism $\alpha: \M_0 \to \M'$, where $\M_0 \subset \D_0$ is the strongly divisible lattice corresponding to $T_0$.

Let $D_0(1)$ be the filtered $\varphi$-module with $\Fil^i (D_0(1) \otimes K) := \Fil^{i-1}(D_0 \otimes K)$ and $\varphi_{D_0(1)}: = p\varphi_{D_0}$. Consider the map $f: D' \to D_0(1)$, where $x \mapsto N_D(x)$, then it is a morphism in the category of filtered $\varphi$-modules. It induces $V_0(1) \to V'$, and by Corollary 3.5.5 in \cite{Bre02}, $T_0(1)$ maps into $T'$ under the morphism. By Proposition \ref{ThuaS}, it induces a morphism between Kisin modules $\huaM' \to \huaM_0(1)$. Tensor this morphism with $S$, we get a morphism between Breuil modules $\beta: \M' \to \M_0(1)$. Note that if we take $T_{\textnormal{cris}}$ on this morphism, it is not neccesarily $T_0(1) \to T'$, this is because $\M_0(1)$ is \'{e}tale (thus not unipotent), so we do not know if $T_{\textnormal{cris}}(\M_0(1)) = T_{\huaS}(\huaM(1))$. However, we always have that $T_{\huaS}(\huaM(1)) \to T_{\textnormal{cris}}(\M_0(1))$ is injective by Lemma 5.3.1 of \cite{Liu07}. And if we tensor $\M' \to \M_0(1)$ with $\frac{1}{p}$, we will get back to $f\otimes Id_S: \D' \to \D_0(1)$. This is because the injective map in the lower left corner of the following commutative diagram will become bijective after tensoring with $\frac{1}{p}$:

$$
\xymatrix{
T_0(1) \ar[d]^{=} \ar[r] &T'\ar[d]^{=}\\
T_{\huaS}(\huaM_0(1)) \ar@{^{(}->}[d]  \ar[r] \ar[d] &T_{\huaS}(\huaM') \ar[d]^{=}\\
T_{\textnormal{cris}}(\M_0(1))  \ar[r] &T_{\textnormal{cris}}(\M')
}
$$

Similarly as in the proof of Corollary 3.5.7 of \cite{Bre02}, if we composite $\beta: \M' \to \M_0(1)$ with $\alpha(1): \M_0(1) \to \M'(1)$, we will get $\M' \to \M'(1)$. The morphism is in fact $1\otimes N_D$, which tells us that $\M$ is stable under $1\otimes N_D$. Since $\M$ is stable under $N_S\otimes 1$ by (3) of Lemma \ref{collect}, so $\M$ is stable under $N_{\D}= 1\otimes N_D + N_S\otimes 1$, and this implies the essential surjectivity of $T_{\textnormal{st}}$.

\end{proof}

\bibliographystyle{alpha}

\end{document}